\documentclass[12pt,a4paper,leqno]{amsart}

\usepackage[top=1.1in, bottom=1.3in, left=1.5in, right=1.5in]{geometry}

\usepackage[ansinew]{inputenc}
\usepackage[T1]{fontenc}
\usepackage[english]{babel}
\usepackage{amsthm}
\usepackage{amsfonts}         
\usepackage{amsmath}
\usepackage{amssymb}
\usepackage{breqn}
\usepackage{bm}
\usepackage{url}
\usepackage{esint}
\usepackage{fancyhdr}

\newcommand{\R}{\mathbb{R}}
\newcommand{\C}{\mathbb{C}}
\newcommand{\Q}{\mathbb{Q}}

\newcommand{\Z}{\mathbb{Z}}

\newcommand{\F}{\mathbb{F}}

\newcommand{\A}{\mathcal{A}}
\newcommand{\B}{\mathcal{B}}
\newcommand{\CC}{\mathcal{C}}

\renewcommand{\SS}{\mathcal{S}}
\newcommand{\PP}{\mathbb{P}}

\newcommand\pprec{\prec\mkern-5mu\prec}

\renewcommand{\d}{\mathfrak{d}}
\newcommand{\D}{\mathcal{W}}
\newcommand{\DD}{\mathcal{D}}
\newcommand{\U}{\mathcal{U}}
\renewcommand{\P}{\mathcal{P}}
\renewcommand{\epsilon}{\varepsilon}

\theoremstyle{plain}
\newtheorem{theorem}{Theorem}
\newtheorem{lemma}[theorem]{Lemma}
\newtheorem{prop}[theorem]{Proposition}
\newtheorem{cor}[theorem]{Corollary}
\newtheorem{definition}{Definition}

\newtheorem{ex}{Example}

\theoremstyle{remark}
\newtheorem{remark}{Remark}

\numberwithin{equation}{section}
\begin{document}
\title[Outside primes of elliptic curves]{On totally split primes in high-degree torsion fields of elliptic curves}

\author{Jori Merikoski}

\address{Mathematical Institute,
University of Oxford,
Andrew Wiles Building,
Radcliffe Observatory Quarter,
Woodstock Road,
Oxford,
OX2 6GG}
\email{jori.merikoski@maths.ox.ac.uk}

\begin{abstract} Analogously to primes in arithmetic progressions to large moduli, we can study primes that are totally split in extensions of $\Q$ of high degree. Motivated by a question of Kowalski we focus on the extensions $\Q(E[d])$ obtained by adjoining the coordinates of $d$-torsion points of a non-CM elliptic curve $E/\Q$.  We show that for almost all integers $d$ there exists a non-CM elliptic curve $E/\Q$ and a prime $p<|\text{Gal}(\Q(E[d])/\Q)|= d^{4-o(1)}$ which is totally split in $\Q(E[d])$. Note that  such a prime $p$ is not accounted for by the expected main term in the Chebotarev Density Theorem. Furthermore, we prove that for almost all $d$ that factorize suitably there exists a non-CM elliptic curve $E/\Q$ and a prime $p$ with $p^{0.2694} < d$ which is totally split in $\Q(E[d])$. 

To show this we use work of Kowalski to relate the question to the distribution of primes in certain residue classes modulo $d^2$. Hence, the barrier $p < d^4$ is related to the limit in the classical Bombieri-Vinogradov Theorem. To break past this we make use of the assumption that $d$ factorizes conveniently, similarly as in the works on primes in arithmetic progression to large moduli by Bombieri, Friedlander, Fouvry, and Iwaniec, and in the more recent works of Zhang, Polymath, and the author. In contrast to these works we do not require any of the deep exponential sum bounds (i.e. sums of Kloosterman sums or Weil/Deligne bound). Instead, we only require the classical large sieve for multiplicative characters and we apply Harman's sieve method to obtain a combinatorial decomposition for primes.
\end{abstract}

\maketitle
\tableofcontents

\section{Introduction}
One of the main topics in modern analytic number theory is the study of primes in arithmetic progressions to large moduli. By analogy, we can investigate primes that are totally split in some extensions $K/\Q$ of high degree. To construct such extensions with additional structure we adjoin to $\Q$ torsion points of elliptic curves $E/\Q$. Given such an elliptic curve and an integer $d\geq 1$, we let $\Q(E[d])$ denote the finite Galois extension of $\Q$ obtained by adjoining the coordinates of $E[d](\bar{\Q})$, the $d$-torsion points of $E(\bar{\Q})$ (i.e. $\Q(E[d])$ is the smallest extension $L/\Q$ such that $E[d](\bar{\Q}) \subseteq E(L)$). Inspired by Kowalski's work \cite{kowalski} we study primes $p$ which are totally split in $\Q(E[d])$ for some $E$, where $d$ is very large in terms of $p$.

To motivate the statement of our results recall that by the Chebotarev Density Theorem for a fixed $d$ we have
\begin{align} \label{chebotarev}
\sum_{\substack{p \leq X \\ p \, \text{totally split in} \, \Q(E[d])}} 1 \,\sim \,\frac{1}{|G_d|} \text{li}(X),
\end{align}
where $G_d$ is the Galois group of $\Q(E[d])$ over $\Q$. By \cite[Theorem 2.1]{kowalski} (originally due to Deuring and Serre) we know that $G_d$ is of size $d^{g-o(1)}$, where $g=2$ if $E$ has complex multiplication and $g=4$ for a non-CM curve $E$ (recall that $E$ is said to have complex multiplication (CM) if its endomorphism ring $\text{End}(E) \supseteq \Z$ is strictly larger than $\Z$). The asymptotic (\ref{chebotarev}) is conjectured to hold uniformly up to $d \leq X^{1/g-\epsilon}$, and assuming GRH for Artin $L$-functions the asymptotic is known to hold up to $d \leq X^{1/2g-\epsilon}$ \cite[Proposition 3.6]{kowalski}. 

The analogy with primes in arithmetic progressions to large moduli is most evident from considering the cyclotomic extensions $\Q(\zeta_d)$, where $\zeta_d$ denotes a primitive $d$-th root of unity. Then prime $p$ is totally split in $\Q(\zeta_d)$ if and only if $p\equiv 1 \, (\text{mod} \, d)$. Therefore, in this case the Chebotarev Density Theorem reduces to the Prime Number Theorem in arithmetic progressions 
\begin{align*}
\sum_{\substack{p \leq X \\ p \, \text{totally split in} \, \Q(\zeta_d)}} 1 = \sum_{\substack{p \leq X \\ p\equiv 1 \, (\text{mod} \, d) }} 1= \pi(X;d,1) \sim \frac{1}{\varphi(d)} \text{li}(X).
\end{align*}
Note that if a prime $p$ is totally split in $\Q(E[d])$, then $p \equiv 1 \,\, (d)$ \cite[Corollary 6.12]{kowalski}.

Compared to primes in arithmetic progressions, a new phenomenon occurs in the case of totally split primes in $\Q(E[d])$ for a non-CM elliptic curve $E$. For $X < |G_d|=d^{4-o(1)}$ the expected main term in (\ref{chebotarev}) is $<1$. However, if we vary $d$ or the elliptic curve $E$ we still expect to find some non-CM curves such that the sum is non-zero, even up to $X=d^{2+o(1)}$ (cf. \cite[Example 3.16]{kowalski} for a connection to the conjecture that there are infinitely many primes of the form $n^2+1$). This is in contrast to the classical case where $\pi(X;d,1) =0$ as soon as $d > X$. Similarly as in \cite[Section 3.4]{kowalski}, we make the following definition.

\begin{definition}
We say that a prime $p$ is an outside prime of a non-CM curve $E/\Q$ if for some $d$ it is totally split in $\Q(E[d])$ and $p<|G_d|$.
\end{definition}

\begin{ex}\emph{(cf. \cite[Section 7.2]{kowalski})}. Consider the curve
\begin{align*}
E:  \quad y^2 = x^3 + 6x -2.
\end{align*}
The smallest outside prime for this curve is $p=196561$, which occurs for $d=140.$ The outside primes for a fixed curve form a very sparse set.  There are only $10$ outside primes of $E$ below $300,000,000$.
\end{ex}

To study the outside primes we consider the problem for almost all $d$ with varying $E$ (for a fixed non-CM curve $E$ the problem appears extremely difficult). We will later give quantitative versions of our main results once we have defined the necessary notations. For now we can state the following theorem, which is a corollary of our Theorem \ref{almostallmodulitheorem}.
\begin{theorem} \label{leasttheorem} Let $D_1$ be sufficiently large and let $\epsilon>0$ be small. For all but  $O(\sqrt{\epsilon}D_1)$ of integers $d \in [D_1,2D_1]$ there exists a prime $p$ with $p < d^{4-\epsilon}$ and a non-CM elliptic curve $E$ such that $p$ is totally split in $\Q(E[d])$.  That is, $p< |G_d|^{1-\epsilon/4 + o(1)}$ so that such a $p$ is an outside prime of $E$. 
\end{theorem}
\begin{remark}
Assuming GRH for Artin $L$-functions we have (\ref{chebotarev}) in the form
\begin{align*}
\sum_{\substack{p \leq X \\ p \, \text{totally split in} \, \Q(E[d])}} 1 \,= \,\frac{1}{|G_d|} \text{li}(X) + O(\sqrt{X} \log (d N_E X)),
\end{align*}
where the $O$-constant is absolute  (cf. \cite[Proposition 3.6]{kowalski}). Note that the error term depends  on $N_E$, the conductor of $E$. From the proof we will see that the curve $E$ in Theorem \ref{leasttheorem} is a lift of some $E/\F_p$, so that we have the trivial bound $N_E \ll p^3 \leq X^3$, using the fact that $N_E$ divides the minimal discriminant. Hence, the existence of outside primes given by Theorem \ref{leasttheorem} is not an artefact of elliptic curves with an exponentially large conductor $N_E$, which would make the secondary term $O(\sqrt{X} \log (d N_E X))$ possibly very large.
\end{remark}
To prove Theorem \ref{leasttheorem} we will have to assume that $d$ has a suitable factorization $d=rq$ for some prime $q$ which is not too small. Due to this the density of the exceptional set of $d$ is $O(\sqrt{\epsilon})$. By computing the dependence on $\epsilon$ explicitly we get a result also for larger $d$. To state it, for any $\alpha \in [1/4,1/2)$ and for any small $\eta>0$ define the set of well-factorable moduli
\begin{align*}
\D_\eta(\alpha,D_1):= \{ d \in [D_1,2D_1] : \, d=rq, \, q \in \PP, \, r \in [D_1^{1/(2\alpha)-1-2\eta},D_1^{1/(2\alpha)-1-\eta}]\}.
\end{align*}
Note that this set contains a positive proportion of all integers in $[D_1,2D_1]$. Requiring a suitable factorization is motivated by developments on primes in arithmetic progressions beyond the Bombieri-Vinogradov range where one also has to restrict the moduli  (cf. \cite{BFI,FI,polymath,zhang}, for instance). In our application it is crucial that $d$ has a factor $r$ which is a bit smaller than $p^{1/2}/d$.

\begin{theorem}\label{main2theorem}
There exists a small constant $\eta >0$ such that for all but a proportion $O_C(\log^{-C} D_1)$ of integers $d\in \D_\eta(0.2694,D_1)$ there exists a prime $p$ with $p^{0.2694}<d$ and a non-CM elliptic curve $E$ such that $p$ is totally split in $\Q(E[d])$.
\end{theorem}

To approach this problem we rely  on the following characterization of totally split primes in $\Q(E[d])$. Suppose that $E/\Q$ has good reduction modulo $p$. By \cite[Lemma 2.2]{kowalski} the group structure of the reduction modulo $p$ is
\begin{align} \label{d1definition}
E_p(\F_p )\simeq \Z/d_1\Z \oplus \Z/d_1d_2\Z 
\end{align}
for certain  unique integers $d_1=d_1(p,E)$ and $d_2=d_2(p,E)$. Then by \cite[Lemma 2.7]{kowalski} $p$ is totally split in $\Q(E[d])$ if and only if $d|d_1(p,E)$. Hence, the problem is equivalent to finding elliptic curves $E/\F_p$ that have exceptionally large $d_1(p)$ (we can always lift $E/\F_p$ to an elliptic curve over $\Q$ and by considering shifts of the defining equation by $kp$ with $k \in \Z$ we can get a non-CM curve with the same reduction modulo $p$). The question of possible group structures for $E_p$ (i.e. possible values of $d_1$ and $d_2$) has been studied in \cite{bpfs,cdks,frs} but all of these works are restricted to the case $d_1 < p^{1/4}$. The formulation of the problem is slightly different here, so that we do not really extend the results in \cite{bpfs,cdks,frs}.

Using the above characterization we also get the following corollary of Theorem \ref{main2theorem}.
\begin{cor} For infinitely many primes $p$ there exists an elliptic curve $E/\F_p$ such that $d_1(E) > p^{0.2694}$.
\end{cor}

\subsection{Previous results}
The study of outside primes was initiated by Kowalski \cite{kowalski}. Similarly as in \cite[Section 6]{kowalski}, we use the following argument to reduce the problem to a question about primes in arithmetic progressions. By Hasse's bound we know that
\begin{align*}
|E_p(\F_p)| = p+1-a_p(E)
\end{align*}
for some $|a_p(E)| < 2 \sqrt{p}$. Then, by \cite[Lemma 2.3]{kowalski} we have
\begin{align*}
p \equiv a_p(E) - 1 \, (\text{mod} \, d_1^2) \quad \text{and} \quad a_p(E) \equiv 2 \, (\text{mod} \, d_1).
\end{align*}
Since we are allowed to vary the elliptic curve $E$, we are led to a problem of finding primes $p \equiv a-1 \, (\text{mod} \, d^2)$ for some $a \equiv 2 \, (d)$ with $|a| < 2 \sqrt{p}$, where $d^2 > p^{1/2 + \epsilon}$. This is made precise in the following two lemmata, the first of which is standard and the latter of which is based on theory developed by Deuring, Honda, Tate, and Waterhouse. In the below $d_1=d_1(E)=d_1(p,E)$ is as in (\ref{d1definition}).
\begin{lemma}\cite[Lemma 2.7]{kowalski} Let $E/\Q$ be an elliptic curve and let $d \geq 1$. Let $p$ be a prime such that $E$ has good reduction modulo $p$. Then $p$ is totally split in $\Q(E[d])$ if and only if $d|d_1(p,E)$.
\end{lemma}
\begin{lemma}\cite[Corollary 6.12]{kowalski} \label{kowalskilemma} Let $p$ be a prime and let $d \geq 1$ be an integer. Then there exists an elliptic curve $E/\F_p$ with $d_1(E)=d$ if and only if there exists an integer $a$ with $|a|<2 \sqrt{p}$ such that $a \equiv 2 \, (d)$ and $d^2| p+1-a$. Furthermore, for $d>2p^{1/4}$ such an intenger $a$ is unique if it exists.
\end{lemma}
To see the uniqueness of $a$ in the above lemma, note that (for $d>2p^{1/4}$ ) if there are two such integers $a_1$ and $a_2$, then the conditions $|a_j| < 2 \sqrt{p}$, $d^2>4 \sqrt{p}$, and $p \equiv a_1-1 \equiv a_2-1 \, (d^2)$ imply that $a_1=a_2$.
 
To study the values of $d_1(p,E)$, in \cite[Section 6.3]{kowalski} the following sets are defined
\begin{align*}
\DD_1(p)&:= \{d\geq 1: \,\text{there exists } E/\F_p \, \text{with} \, d=d_1(E)\}, \\
\DD_s(p)&:= \DD_1(p)\cap \{d \leq 2 p^{1/4}\}, \quad \text{and} \quad \DD_\ell(p):= \DD_1(p)\cap \{d > 2p^{1/4}\}.
\end{align*}
Then by Lemma \ref{kowalskilemma} we see that $d \in \DD_s(p)$ if and only if $d|(p-1)$ and $d \leq 2 p^{1/4}$, since in that case we can always find a suitable $a$ \cite[Lemma 6.36]{kowalski}. For $d>2p^{1/4}$ the integer $a$ is unique if it exists. Therefore, we have
\begin{align*} \\
\sum_{p \leq X} |\DD_s(p)|\, &= \sum_{d \leq 2 X^{1/4}} \sum_{\substack{d^4 /16 < p \leq X \\ p \equiv 1 \,(d)}} 1 \quad \quad\quad \quad\quad \quad \text{and} \\
\sum_{p \leq X} |\DD_\ell(p)|\, &= \sum_{d \geq 1} \sum_{\substack{a < 2\sqrt{2X} \\ a \equiv 2 \, (d)}} \,\sum_{\substack{|a|^2/4 < p \leq \min\{X,d^4/16\} \\ p \equiv a-1 \, (d^2)}} 1.
\end{align*}

Hence, understanding the possible values of $d_1(p,E)$ reduces to studying distribution of primes in arithmetic progressions. By using the Bombieri-Vinogradov Theorem and the Brun-Titchmarsh inequality, it was shown by Kowalski \cite[Propositions 6.39 and 6.40]{kowalski} that  for $c=\zeta(2)\zeta(3)/\zeta(6)$ 
\begin{align*}
\sum_{p \leq X} |\DD_s(p)| = \frac{c X}{4} + O\bigg(\frac{X}{\log X}\bigg) \quad
\text{and} \quad \sum_{p \leq X} |\DD_\ell(p)| \ll \frac{X}{\log X},
\end{align*}
which together imply \cite[Proposition 6.38]{kowalski}
\begin{align*}
\sum_{p \leq X} |\DD_1(p)| = \frac{c X}{4} + O\bigg(\frac{X}{\log X}\bigg).
\end{align*}
In \cite[Remark 6.44]{kowalski} it is noted that for heuristics about outside primes it would be helpful to have a lower bound for $\sum_{p\leq X} |\DD_\ell(p)|$ and that this appears quite difficult since we have to understand the distribution of primes in arithmetic progressions to moduli $d^2 > 4 \sqrt{p}$, which goes beyond the Bombieri-Vinogradov range. Answering this question was the original motivation for this manuscript. To the author's knowledge this problem and outside primes in general have not been addressed in the literature outside of \cite{kowalski}.

It was recently shown by the author \cite{m} that for any fixed $b \neq 0$ there are infinitely many primes $p$ such that there exists a large square divisor $d^2|(p-b)$ with $d^2 > p^{1/2+1/2000},$ which improved the previous results of Matom\"aki \cite{matomaki}, and Baier and Zhao \cite{bz}. To break past the $p^{1/2}$-barrier it is required that $d$ satisfies a very strong factorization property, namely that all prime factors of $d$ are less than $d^\delta$ for some small $\delta >0$. This already implies the above results for some $\epsilon$ with a fixed $a_p(E)=2$ and restricting $d$ to $d^\delta$-smooth numbers. The arguments in this manuscript are similar in spirit but the extra averaging over $a$ offers major simplifications to the arguments as well as a much better exponent.
\subsection{Asymptotic results} \label{asympsubsection}
Motivated by the previous discussion, for any $X\gg 1$ and $\theta \in (1/2,1)$  we define
\begin{align*}
\DD_X(p;\theta) := \{ X^{\theta} < d^2 \leq 2 X^{\theta}:  \, d = d_1(E) \quad \text{for some} \quad E / \F_p \}.
\end{align*}
We will show
\begin{theorem} \label{asymptotictheorem} For $\theta=1/2 + \epsilon $ with small $\epsilon>0$ we have
\begin{align*}
\sum_{X < p \leq 2X} |\DD_X(p;\theta)| = \sum_{X^\theta < d^2 \leq 2 X^\theta} \sum_{\substack{|a| < 2 \sqrt{2 X} \\ a \equiv 2 \,\, (d)}} \frac{2X- a^2/4}{\varphi(d^2) \log X} + O \bigg( \frac{ \sqrt{\epsilon} X^{3/2-\theta}}{\log X}  \bigg),
\end{align*}
where the implied constant is absolute.
\end{theorem}
\begin{remark} The sum on the right-hand side can be computed explicitly to get
\begin{align*}
\sum_{X^\theta < d^2 \leq 2 X^\theta} \sum_{\substack{|a| < 2 \sqrt{2 X} \\ a \equiv 2 \,\, (d)}} \frac{2X- a^2/4}{\varphi(d^2) \log X} \,= \, (1+o(1))\frac{1575 \zeta(3) \sqrt{2}}{4 \pi^4} \cdot \frac{X^{3/2-\theta}}{\log X}.
\end{align*}
\end{remark}
We will prove this in Section \ref{sieveasympsection}. The main idea is to restrict to $d$ which factorize suitably. For $\theta=1/2+\epsilon$ we will be able to work with a set of $d$ with density $1-O( \sqrt{\epsilon})$.

By a slight modification of our arguments we obtain an asymptotic formula for the summatory function of $|\DD_\ell(p)|$, answering the question of Kowalski \cite[Remark 6.44]{kowalski}. To see this, choose $\epsilon=10 \log \log X/ \log X$ and apply similar arguments in the range $d^2 \in [X^{1/2-\epsilon}, X^{1/2+\epsilon}]$, splitting the sum dyadically. The contribution from $d^2> X^{1/2+\epsilon}$ is $\ll X \log^{-10} X$ by \cite[Proposition 6.42]{kowalski}, for instance. The contribution from $d^2< X^{1/2-\epsilon}$ is negligible by trivial bounds. For the middle range we are able to save a factor of $\sqrt{\epsilon}$ in the error terms, similar to Theorem \ref{asymptotictheorem}. The precise details are simple enough but fairly lengthy, so we leave the proof to the interested reader.
\begin{theorem} \label{delltheorem}
We have
\begin{align*}
\sum_{p \leq X} |\DD_\ell (p)| \, = \sum_{d \geq 1} \sum_{\substack{a < 2\sqrt{X} \\ a \equiv 2 \, (d)}} \frac{1}{\varphi(d^2)}\sum_{\substack{|a|^2/4 < p \leq \min\{X,d^4/16\} }} 1+ O \bigg( \frac{X (\log \log X)^{1/2}}{\log^{3/2} X}\bigg),
\end{align*}
where the implied constant is absolute.
\end{theorem}

\subsection{Lower bound result}
If we are not interested in asymptotic formulas but lower bounds of the correct order of magnitude, then the exponent $\theta$ can be increased considerably. The proof of the following result is given in Section \ref{lowerboundsection}.
In the proof we make use of a certain factorization property $d=rq$ (cf. Section \ref{lowerboundsection}) and the fact that a positive proportion of $d$ satisfy this property. 
\begin{theorem} \label{lowerboundtheorem} Let $\theta \in [1/2,1/2+0.0388]$. Then
\begin{align*}
\sum_{X < p \leq 2X} |\DD_X(p;\theta)| \,\, \gg  \frac{X^{3/2-\theta}}{ \log X}.
\end{align*}
\end{theorem}

\subsection{Almost all moduli}
Our proofs of Theorems \ref{asymptotictheorem} and \ref{lowerboundtheorem} actually give a stronger result in the sense that the bounds hold point-wise for almost all $d$ of a suitable form. Define
\begin{align*}
\P(d):=\{p \in \PP: d=d_1(p,E) \text{ for some } E/\F_p \}.
\end{align*}
Theorem \ref{leasttheorem} is then a direct corollary of the following theorem, which is proved in Section \ref{almostallsection}.
\begin{theorem} \label{almostallmodulitheorem}  Let $\epsilon>0$ be small. For all but  $O(\sqrt{\epsilon}X^{1/4+\epsilon})$ of integers $d \in [X^{1/4+\epsilon},2X^{1/4+\epsilon}]$ we have
\begin{align*}
\sum_{X< p \leq 2X} 1_{p \in \P(d)} = \sum_{\substack{|a| < 2 \sqrt{2 X} \\ a \equiv 2 \,\, (d)}} \frac{2X- a^2/4}{\varphi(d^2) \log X} +  O \bigg( \frac{ \sqrt{\epsilon} X^{3/2}}{d \varphi(d^2) \log X}  \bigg),
\end{align*}
where the implied constant is absolute.
\end{theorem}

Similarly, Theorem \ref{main2theorem} follows directly from the following. Note that the lower bound  is of the correct order of magnitude.
\begin{theorem} \label{almostall2theorem}
There exists a small constant $\eta>0$ such that for all but a proportion $O_C(\log^{-C} X)$ of integers $d \in  \D_\eta(0.2694,X^{0.2694})$ we have
\begin{align*}
\sum_{X< p \leq 2X} 1_{p \in \P(d)}\, \gg \frac{X^{3/2}}{d\varphi(d^2) \log X}.
\end{align*}
\end{theorem}

\subsection{Outline of the paper}
The article is organized so that in Section \ref{sieveasympsection} we give the proof of the asymptotic result (Theorem \ref{asymptotictheorem}), while in Section \ref{lowerboundsection} we show the lower bound version (Theorem \ref{lowerboundtheorem}). In Section \ref{almostallsection} we show how to modify the arguments to get Theorems \ref{almostallmodulitheorem} and \ref{almostall2theorem}.

Similarly as in \cite{matomaki} and in our previous work \cite{m}, we use Harman's sieve method to detect primes (cf. Harman's book \cite{harman}). Roughly speaking, Harman's sieve method is a technique for obtaining a combinatorial decomposition for a sum $S(\A):=\sum_{p \sim X} f_p$ of primes along a given non-negative sequence $(f_n)_{n\geq 1}$. We can prove a lower bound for $S(\A)$ provided that we have asymptotic formulas for the so-called Type I and Type II sums
\begin{align*}
\sum_{\substack{m \sim M \\ n \sim N}} \alpha(m) f_{mn} \quad \text{and} \quad \sum_{\substack{m \sim M \\ n \sim N}} \alpha(m) \beta(n) f_{mn}
\end{align*}
with arbitrary bounded coefficients $\alpha(m)$ and $\beta(n)$ for $MN=X$ with $M$ and $N$ in some suitably wide ranges. The reader may be more familiar with the classical Vaughan identity which states that we can get an asymptotic formula for $S(\A)$ provided that we can compute the Type I and Type II sums in a wide range (e.g. Type I sums  with $M\leq X^{1/4}$ and Type II sums with $X^{1/4} \leq M,N \leq X^{3/4}$). The benefit of using Harman's sieve method is that we can use positivity to regard some of the ranges as error terms so that we still get lower and upper bounds for $S(\A)$ even if we are not able to handle Type II sums in all ranges (e.g. in Section \ref{sieveasympsection} we have asymptotics for Type II sums in the range $X^{1/4+O(\delta)} \ll N \ll X^{1/2-O(\delta)}$ for some small $\delta>0$ which is almost sufficient for the Vaughan identity). The exponent 0.2694 in Theorem \ref{leasttheorem} is then determined by a numerically computed upper bound for these error terms in Section \ref{lowerboundsection}.

The heart of the matter is the handling of Type I and Type II sums which is carried out in Sections \ref{typeisection} and \ref{typeiisection} (Propositions \ref{prelitypeiprop} and \ref{prelitypeiiprop}). As usual, with the Type I sums we already have a smooth summation over $n \sim N$ on which to apply Fourier analysis (Poisson summation), so that these are easily computed. 

For Type II sums we must first apply Cauchy-Schwarz to smoothen one of the coefficients. Since we are interested in results for almost all moduli $d$, it suffices to prove Type II information for almost all $d$, that is, give non-trivial bounds for sums of the form
\begin{align*}
\sum_{d \sim D}& \bigg|\sum_{\substack{m \sim M \\ n \sim N}} \alpha(m) \beta(n)( f_{mn}(d)-EMT_{mn}(d))\bigg|
\end{align*}
where $EMT_{mn}(d)$ is the expected average value of our sequence $f_{mn}(d)$ which now depends on $d$ also. Similarly as in the results on primes in arithmetic progressions to large moduli \cite{BFI,FI,polymath,zhang}, we make use of the fact that our modulus $d=rq$ factorizes suitably so that we may rearrange the sum into
\begin{align*}
\sum_{d=rq}& \bigg|\sum_{\substack{m \sim M \\ n \sim N}} \alpha(m) \beta(n)( f_{mn}(rq)-EMT_{mn}(rq))\bigg| \\
&= \sum_r \sum_{m } \alpha(m) \sum_q c_{rq} \sum_{n} \beta(n) ( f_{mn}(rq)-EMT_{mn}(rq)).
\end{align*}
for some bounded coefficients $c_{rq}$. We apply Cauchy-Schwarz with the sums over $n$ and $q$ on the inside to replace $\alpha(m) $ with $1$, where the sum over $q$ helps to bound the diagonal contribution. Expanding the square we get
\begin{align*}
\sum_r \sum_{m \sim M} \bigg| \sum_q c_{rq} \sum_{n} \beta(n) ( f_{mn}(rq)-EMT_{mn}(rq))\bigg|^2 =: W-2V+U.
\end{align*}
We then evaluate each of the sums $W,V,U$ separately to get
\begin{align*}
U,V,W=X_0(1+O_C(\log^{-C}X))
\end{align*}
 for some quantity $X_0$, so that the main terms cancel in $W-2V+U \ll_C X_0 \log^{-C}X$ (this is also known as Linnik's dispersion method). As usual, the hardest term to evaluate is
\begin{align*}
W =\sum_r \sum_{q_1,q_2} c_{rq_1}c_{rq_2} \sum_{n_1,n_2} \beta(n_1)\beta(n_2) \sum_{m \sim M} f_{mn_1}(rq_1)f_{mn_2}(rq_2).
\end{align*}
For primes in arithmetic progressions this is achieved by using the heavy machinery of sums of Kloosterman sums \cite{BFI,FI} or the Deligne/Weil bounds for algebraic exponential sums \cite{polymath,zhang}. We are reprieved by the fact that in the definition of $f_n$ we have an average over the residue classes $a < 2\sqrt{p}, \, a \equiv 2 \, \, (d)$, which allows us to rewrite $W$ in such a way that the estimation can be done simply by using the classical large sieve for multiplicative characters. It is noteworthy that our results do not rely on any of the deeper bounds, not even on the special case of Weil bound for Kloosterman sums. It should also be noted that the averaging over residue classes $a$ also allows us to control the diagonal contribution in the Cauchy-Schwarz step, so that despite our moduli running over the sparse set of squares we can handle Type II sums in a fairly wide range of $M$ and $N$ (this is in contrast to \cite[Proposition 4]{m} where the sparseness of the moduli set results in a very narrow range for the Type II sums). 

\begin{remark} The exponent $0.2694$ in Theorem \ref{leasttheorem} is not necessarily optimal. One can certainly improve this slightly by a more careful optimization of the sieve argument in Section \ref{lowerboundsection}. It is also likely that we could strengthen the arithmetical information (Propositions \ref{prelitypeiprop} and \ref{prelitypeiiprop}) that is used in the sieve (for example, by incorporating ideas from \cite{m}). We have chosen not to pursue these issues here in order to keep the arguments as simple as possible and so that our results are independent of the harder exponential sum bounds.
\end{remark}

\begin{remark} Our arguments rely heavily on sieve methods. For this we need to extend functions such as $p \mapsto |\DD_X(p;\theta)|$ from primes to all integers. It is not clear to the author if this or any of the arguments below have a natural interpretation purely in terms of elliptic curves. It would of course be very interesting if such an interpretation were to exist. It would also be interesting to see if there is an explanation of the factorization $d_1(p)=rq$ in terms of $E/\F_p$ which would explain why curves of this form appear to be easier to handle (as opposed to $d_1(p) \in \PP$, for instance).
\end{remark}

\subsection{Notations}
For functions $f$ and $g$ with $g$ positive, we write $f \ll g$ or $f= \mathcal{O}(g)$ if there is a constant $C$ such that $|f|  \leq C g.$ The notation $f \asymp g$ means $g \ll f \ll g.$ The constant may depend on some parameter, which is indicated in the subscript (e.g. $\ll_{\epsilon}$).
We write $f=o(g)$ if $f/g \to 0$ for large values of the variable. For variables we write $n \sim N$ meaning $N<n \leq 2N$. 

It is convenient for us to define
\begin{align*}
A \pprec B
\end{align*}
to mean $A \, \ll_\epsilon X^{\epsilon} B$ for any $\epsilon>0.$  A typical bound we use is $\tau_k(n) \pprec 1$  for $n \ll X$, where $\tau_k$ is the $k$-fold divisor function. We say that an arithmetic function $f$ is \emph{divisor bounded} if $|f(n)| \ll \tau_k(n)$ for some $k$.

We let $\eta >0$ denote a sufficiently small constant, which may be different from place to place. For example, $A \ll X^{-\eta}B$ means that the bound holds for some $\eta >0, $ and $A \pprec X^{-\eta}B$ is equivalent to $A \ll X^{-\eta}B$.

For a statement $E$ we denote by $1_E$ the characteristic function of that statement. For a set $A$ we use $1_A$ to denote the characteristic function of $A.$ 

We also define $P(w):= \prod_{p < w} p,$ where the product is over primes. 

We let $e(x):= e^{2 \pi i x}$ and $e_q(x):= e(x/q)$ for any integer $q \geq 1$. For integers $a,$ $b$, and $q \geq 1$ we define $e_{q}(a/b) := e(a\overline{b}/q)$ if $b$ is invertible modulo $q$, where $\overline{b}$ is the solution to $b\overline{b} \equiv 1 \,\, (q).$

\subsection{Acknowledgements}
I am grateful to my supervisor Kaisa Matom\"aki for support and comments. I am also grateful to Emmanuel Kowalski for suggesting this problem, helpful discussions as well as for hospitality during my visit at ETH Z\"urich in autumn 2018, which was supported by the Emil Aaltonen Foundation. I am also very grateful to the anonymous referees, whose suggestions have greatly clarified the paper. During the work the author was supported by a grant from the Magnus Ehrnrooth Foundation and UTUGS Graduate
School.

\section{Preliminaries}
In this section we state some standard auxiliary results. First is the classical sieve upper bound for primes in arithmetic progressions (cf. \cite[Theorem 7.15]{odc}, for instance).
\begin{lemma}\emph{\textbf{(Brun-Titchmarsh inequality).}}  \label{btlemma} We have 
\begin{align*}
\sum_{\substack{p\sim Y \\ p\equiv a \, (q)}}  1\leq (2 +o(1))\frac{Y}{\varphi(q)\log Y/q}
\end{align*}
uniformly for $q \leq Y^{1-\eta}$ and $a$.
\end{lemma}
We also need the following simple upper bound for smooth numbers (cf.  \cite[Chapter III.5]{Ten}, for example).
\begin{lemma} \label{smoothlemma} For any $2\leq Z \leq Y$ we have
\begin{align*}
\sum_{\substack{n \sim Y }} 1_{P^+(n) < Z} \ll Y \exp(- \log Y/2\log Z).
\end{align*}
\end{lemma}
To compute smoothly weighted sums over arithmetic progressions we need the Poisson summation formula.
\begin{lemma}\emph{\textbf{(Truncated Poisson summation formula).}} \label{poisson}
Let $\psi:\R\to  \C$ be a fixed $C^\infty$-smooth compactly supported function with $\|\psi\|_{1} \leq 1$ and let $N \gg 1$. Fix a real number $x=x(N)$. Let $q \geq 1$ be an integer. Then for any $\epsilon > 0$
\begin{align*}
\sum_{n \equiv a \, (q)} \psi\bigg(\frac{n-x}{N}\bigg) = \frac{1}{q} \sum_{n} \psi\bigg(\frac{n-x}{N}\bigg)+ \frac{N}{q} \sum_{1 \leq |h| \leq N^\epsilon q/N} c_h e_q (-ah) + O_{C,\epsilon}(N^{-C}),
\end{align*}
where $c_h=c_{h,q,\psi,x,N}$ are complex coefficients satisfying $|c_h| \leq 1$.
\end{lemma}
\begin{proof}
By the usual Poisson summation formula
\begin{align*}
\sum_{n \equiv a \, (q)} \psi\bigg(\frac{n-x}{N}\bigg) = \sum_{m} \psi\bigg(\frac{mq+a-x}{N}\bigg) = \sum_{h} \int \psi\bigg(\frac{uq+a-x}{N}\bigg)e(hu) d u.
\end{align*}
Making the change of variables 
\begin{align*}
u \mapsto Nu/q -a/q+x/q
\end{align*}
we get
\begin{align} \label{poissonprelim}
\sum_{n \equiv a \, (q)} \psi\bigg(\frac{n-x}{N}\bigg) = \frac{N}{q} \sum_{h} c_h e_q (-ah),
\end{align}
where
\begin{align*}
c_{h} := e(hx/q)\int \psi(u) e(huN/q) d u.
\end{align*}
For all $h$ we have the trivial estimate $|c_{h}| \leq \|\psi\|_{1} \leq 1$. For $|h| > N^{\epsilon}q/N$ we can iterate integration by parts to show that the contribution from this part is $\ll_{C,\epsilon} N^{-C}.$ As usual, $h=0$ gives us the main term by applying (\ref{poissonprelim}) with $q=1$, noting that all terms except $h=0$ give a negligible contribution.
\end{proof}

Our proof of the Type II estimate relies on the classical large sieve inequality (cf. \cite[Theorem 9.10]{odc}, for instance).
\begin{lemma} \label{largeprimlemma}\emph{\textbf{(Large sieve inequality for multiplicative characters).}}  For any sequence $c_n$ of complex numbers and for any $M,N \geq 1$ we have
\begin{align*}
\sum_{q \leq Q} \frac{q}{\varphi(q)} \sideset{}{'} \sum_{\chi \, \, (q)} \bigg | \sum_{M< n \leq M+N} c_n \chi(n)\bigg |^2 \, \leq \, (Q^2+N) \sum_n |c_n|^2,
\end{align*}
where the sum over $\chi$ is over primitive Dirichlet characters.
\end{lemma}

\begin{definition}\label{swdefinition}\emph{\textbf{(Siegel-Walfisz condition)}} We say that a divisor-bounded arithmetic function $\beta(n)$ satisfies the Siegel-Walfisz condition if for all integers $r,q\geq 1$ with $r\leq N$ and $q = \log^{O(1)} N$ and non-principal characters $\chi$ modulo $q$ we have
\begin{align*}
\sum_{\substack{n \leq N \\ (n,r)=1}} \beta(n)\chi(n) \ll_C \frac{N}{\log^C N}
\end{align*}
for any $C >0.$
\end{definition}
By standard technique Lemma \ref{largeprimlemma} gives us the more practical version where $\chi$ runs over all non-principal characters (cf. \cite[Chapter 9.8]{odc}, for instance).
\begin{lemma} \label{large} Let $N,Q\gg 1$ and let $\beta(n)$  satisfy the Siegel-Walfisz condition. Then
\begin{align*}
\sum_{q \sim Q}  \frac{1}{\varphi(q)}\sum_{\chi  \neq \chi_0 \, \, (q)} \bigg | \sum_{n\sim N} \beta(n) \chi(n)\bigg |^2 \, \ll_C \bigg(Q+\frac{N}{\log^C N} \bigg) N \log^{O(1)}( NQ)
\end{align*}
for any $C>0.$
\end{lemma}

\subsection{Buchstab integrals}
For the sieve arguments we need a lemma for transforming sums over almost-primes into integrals which can be evaluated numerically. Let $\omega(u)$ denote the Buchstab function (cf. \cite[Chapter 1]{harman} for the definition and the properties below, for instance), so that by the Prime Number Theorem for $Y^{\epsilon} < z < Y$ 
\begin{align} \label{buchasymp}
\sum_{Y< n \leq 2Y} 1_{(n,P(z))=1} = (1+o(1)) \omega \left(\frac{\log Y}{\log z} \right) \frac{Y}{\log z}.
\end{align}
In the numerical computations we will use the following upper bound for the Buchstab function (cf. \cite[Lemma 5]{hbjia})
\begin{align*}
\omega(u) \, \leq \begin{cases} 0, &u < 1 \\
1/u, & 1 \leq u < 2 \\
(1+\log(u-1))/u, &2 \leq u < 3 \\
0.5644, &  3 \leq u < 4 \\
0.5617, & u \geq 4.
\end{cases}
\end{align*}
For the lemma below we assume that the range $\,\mathcal{U}\subset [X^{\eta},X]^{k}$ is sufficiently well-behaved (e.g. an intersection of sets of the type $\{ \boldsymbol{u}: u_i < u_j \}$ or $\{\boldsymbol{u}: V <  f(u_1, \dots,u_k) < W\}$ for some polynomial $f$ and some fixed $V,W$). This condition will be clear from the context everywhere below. In short, we need to be able to pass from a sum over primes in $\mathcal{U}$ to  a sum over integers (by the Prime number theorem), and from the sum over integers to the corresponding integral.

\begin{lemma} \label{bilemma} Let $\mathcal{U} \subset [X^{\eta},X]^{k}$ be a sufficiently well-behaved set and such that for all $(u_1,\dots,u_k) \in \U$ we have $u_1\cdots u_k < X^{1-\eta}$. Then for any fixed integer $k \leq X$ we have
\begin{align*}
\sum_{(p_1, \dots , p_k) \in \mathcal{U}} \sum_{\substack{q \sim X/p_1\cdots p_k }}  1_{(q,P(p_k))=1} 1_{(q,k)=1} = (1+o(1))\frac{X}{\log X} \int \omega (\boldsymbol{\alpha }) \frac{d\alpha_1 \cdots d\alpha_k}{\alpha_1\cdots\alpha_{k-1}\alpha_k^2},
\end{align*}
where the integral is over the range 
\begin{align*}
\{\boldsymbol{\alpha}: \, (X^{\alpha_1}, \dots, X^{\alpha_k}) \in \mathcal{U}\}
\end{align*}
 and $\omega(\boldsymbol{\alpha})= \omega(\alpha_1,\dots,\alpha_k):= \omega((1-\alpha_1-\cdots 
 -\alpha_k)/\alpha_k)$.
\end{lemma}
\begin{proof}
Since $(q,P(X^\eta))=1$, we may drop the condition $(q,k)=1$ with a negligible error term.  By (\ref{buchasymp}) and by the Prime Number Theorem, the left-hand side is
\begin{align*}
& (1+o(1))X \sum_{(p_1, \dots , p_k) \in \mathcal{U}} \frac{1}{p_1\cdots p_k \log p_k} \omega \left( \frac{\log(X/(p_1\cdots p_k))}{\log p_k} \right) \\
&= (1+o(1))X  \hspace{-5pt}  \sum_{(n_1,\dots,n_k ) \in \mathcal{U}} \frac{1}{n_1\cdots n_k (\log n_1) \dots (\log n_{k-1} )\log^2 n_k} \omega \left( \frac{\log(X /(n_1\cdots n_k))}{\log n_k} \right) \\
&=(1+o(1))X  \int_{\mathcal{U}}  \omega \left( \frac{\log(X/(u_1\cdots u_k))}{\log u_k} \right)   \frac{du_1\cdots du_k}{u_1\cdots u_k (\log u_1) \dots (\log u_{k-1} )\log^2 u_k}\\
&= (1+o(1)) \frac{X}{\log X}  \int \omega (\boldsymbol{\alpha }) \frac{d\alpha_1 \cdots d\alpha_k}{\alpha_1\cdots\alpha_{k-1}\alpha_k^2}
\end{align*}
by the change of variables $u_j=X^{\alpha_j}$.
\end{proof}

\subsection{Finer-than-dyadic decomposition} \label{ftdsection} Here we explain a common device for removal of cross-conditions that does not contaminate coefficients (unlike applying the Mellin transformation). The idea goes back at least to the work of Fouvry and Iwaniec \cite{FI}. As an example, suppose we want a bound of the form
\begin{align} \label{ftdbound}
\sum_{\substack{m \sim M \\ n \sim N \\ m <yn}} a_{m,n}\ll \delta MN
\end{align}
for some $y>0$ and some bounded coefficients $a_{m,n}$ (here $\delta$ is usually either $(MN)^{-\eta}$ or $(\log MN)^{-C}$). To do this we would like to remove the cross-condition $m<yn$, that is, bound the original sum by a linear combination of sums of the same form but without $m<yn$. A standard trick to achieve this without affecting the coefficients $a_{m,n}$ is the so called finer-than dyadic decomposition. As the name suggests, we split the sums over $m$ and $n$ into short intervals to obtain $ \ll \delta^{-2}$ sums of the form
\begin{align*} 
\sum_{\substack{m \in [M_0,(1+\delta)M_0] \\ n \in [N_0,(1+\delta)N_0] \\ m <yn}} a_{m,n}. 
\end{align*}
for $M_0=(1+\delta)^j M \sim M$, $N_0=(1+\delta)^{k} N \sim N$, where $j,k$ are  restricted to
\begin{align*}
(1+\delta)^j M \leq y(1+\delta)^{k+1} N.
\end{align*} 
The condition $m<yn$ is then trivially satisfied except for the boundary cases when
\begin{align*}
y(1+\delta)^{k-1} N \leq (1+\delta)^j M \leq y(1+\delta)^{k+1} N.
\end{align*}
Since there are at most $\ll \delta^{-1}$ of such pairs $j,k$, the contribution from these is trivially bounded by $\delta^{-1}\cdot \delta M \cdot \delta N \ll \delta MN$, as required. Hence, (\ref{ftdbound}) holds if we can show
\begin{align*}
\sum_{\substack{m \in [M_0,(1+\delta)M_0] \\ n \in [N_0,(1+\delta)N_0]}} a_{m,n} \ll \delta^3 MN
\end{align*}
for all $M_0 \asymp M$, $N_0 \asymp N$. By subtracting it suffices to show that
\begin{align*}
\sum_{\substack{m \in [\delta^2 M_1,M_1] \\ n \in [\delta^2 N_1,N_1]}} a_{m,n} \ll \delta^3 MN
\end{align*}
for all $M_1 \asymp M$, $N_1 \asymp N$ (consider this with $(M_1,N_1)$ equal to $(M_0,N_0)$ and $((1+\delta)M_0,(1+\delta)N_0)$ and subtract these to get four sums, three of which are trivial to bound). Splitting these sums, dyadically we need to show
\begin{align*}
\sum_{\substack{m \sim M_2 \\ n \sim N_2 }} a_{m,n} \ll \delta^3 (\log^{-2} 1/\delta) MN
\end{align*}
for all $\delta^2 M \ll M_2 \ll M$, $\delta^2 N \ll N_2 \ll N$, which is of the same form as the original sum but without the cross-condition. In practice we will always have $\delta \gg X^{-\eta}$ and our ranges of  $M$ and $N$ are defined up to factors of $X^\eta$, so that it suffices to show
\begin{align*}
\sum_{\substack{m \sim M \\ n \sim N }} a_{m,n} \ll \delta^3 (\log^{-2} 1/\delta) MN.
\end{align*}
for $M$ and $N$ essentially in the same range as before.

\section{Type I sums} \label{typeisection}
Before the sieve argument we need to gather arithmetic information. We have the following asymptotic formula for the relevant Type I sums. Here the variable $b$ is such that the variable $a$ in the above corresponds to $bd+2$, so that $a \equiv 2 \, (d)$. Note that we let $D=X^\theta$ denote the size of the modulus $d^2$, so that comparing to Theorem \ref{main2theorem} we have $D=D_1^2$.
\begin{prop} \label{prelitypeiprop} Let $D=X^\theta$ for some $\theta \in [1/2, 3/5)$ and let $d^2 \asymp D$. Let $MN\asymp X$ with $M \ll X^{3/2-3\theta/2-\eta}$. Suppose that $B \in [X^{1/2-\eta}D^{-1/2},4 X^{1/2}D^{-1/2}].$  Let $\psi_N(n):=\psi((n-x)/N)$ for some compactly supported $C^{\infty}$-smooth $\psi$ with $\|\psi\|_1 \leq 1$, and some fixed real number $x$ which may depend on $M,N$, and $X$. Then for any divisor bounded $\alpha(m)$ we have
\begin{align*}
 \sum_{b \sim B} \sum_{m \sim M}\alpha(m)  \sum_{\substack{n \\ mn \equiv bd+1 \, \, (d^2)}} \psi_{N}(n) = \frac{1}{\varphi(d^2)}\sum_{b \sim B} \sum_{(m,d^2)=1} \alpha(m) \sum_{(n,d^2)=1} \psi_{N}(n)+  O \bigg( \frac{BX^{1-\eta}}{d^2}\bigg).
\end{align*}

\end{prop}

 \begin{proof}
 We first split the sum over $n$ on the right-hand side into primitive residue classes modulo $d$ and evaluate the sum over $n$ (note that $(n,d^2)=1$ is equivalent to $(n,d)=1$ and we have $N>X^{\eta} d$), which yields
 \begin{align*}
 \frac{1}{\varphi(d^2)}\sum_{b \sim B} \sum_{(m,d^2)=1} \alpha(m) \sum_{(n,d^2)=1} \psi_{N}(n)=\frac{1}{d^2}\sum_{b \sim B} \sum_{(m,d^2)=1} \alpha(m) \sum_{n} \psi_{N}(n) +  O \bigg(\frac{BX^{1-\eta}}{d^2}\bigg).
 \end{align*}
By applying truncated Poisson summation (Lemma \ref{poisson}) to the left-hand side in the claim, we get a matching main term with an error term
 \begin{align*}
 \ll_\epsilon 1+ X^{\epsilon} \hat{\Sigma}(M):=& 1+ \frac{X^{\epsilon}}{H} \sum_{1 \leq |h| \leq H}c_h \sum_{b \sim B} \sum_{(m,d)=1} \alpha(m)e_{d^2} \bigg(-\frac{h(bd+1)}{m} \bigg),
 \end{align*}
where $H:=X^{\epsilon} d^2/N$ and $c_h$ are some bounded complex coefficients. By rearranging and using the Cauchy-Schwarz inequality
\begin{align*}
\hat{\Sigma}(M) &\pprec M^{1/2} \bigg(\sum_{\substack{m \sim M\\(m,d^2)=1}} \bigg|\frac{1}{H}\sum_h\sum_b c_h e_{d^2}\bigg(-\frac{h(bd+1)}{m} \bigg)\bigg|^2 \bigg)^{1/2} \\
&\leq M^{1/2} \bigg( \bigg(1+ \frac{M}{d^2} \bigg)\sum_{\substack{m \,\, (d^2)\\(m,d^2)=1}} \bigg|\frac{1}{H}\sum_h\sum_b c_h e_{d^2} \bigg(-\frac{h(bd+1)}{m} \bigg)\bigg|^2 \bigg)^{1/2} \\
&\leq  M^{1/2} \bigg(\bigg(1+ \frac{M}{d^2} \bigg)\frac{1}{H^2} \sum_{h_1,h_2} \sum_{b_1,b_2} |S(h_1,h_2,b_1,b_2)| \bigg)^{1/2}
\end{align*}
where
\begin{align*}
S(h_1,h_2,b_1,b_2) := \sum_{\substack{m \,\, (d^2)\\(m,d^2)=1}}& e_{d^2} \bigg(\frac{h_1(b_1d+1)-h_2(b_2d+1)}{m} \bigg)
\\ &\pprec (h_1(b_1d+1)-h_2(b_2d+1),d^2)
\end{align*}
by the elementary bound for Ramanujan sums. We have
\begin{align*}
\sum_{h_1,h_2} \sum_{b_1,b_2} (h_1(b_1d+1)-h_2(b_2d+1),d^2) = \sum_{e|d} e \sum_{\substack{h_1,h_2,b_1,b_2 \\ (h_1-h_2,d)=e}} ((h_1(b_1d+1)-h_2(b_2d+1))/e,e) 
\end{align*}
by defining $e=(d,h_1-h_2)$, since $(h_1(b_1d+1)-h_2(b_2d+1),d)=(h_1-h_2,d)$. Write $h_1=h$ and $h_2=h+\ell e$, $|\ell| \ll L:= H/e$. We now bound separately the contribution from the parts $\ell=0$ (where $e=d$) and $\ell \neq 0$ (in which case $(\ell,d/e)=1).$ 

The contribution from $\ell =0$  is bounded by (separating $b_1=b_2$ and $b_1\neq b_2$)
\begin{align*}
&d  \sum_h \sum_{b_1,b_2} ( h(b_1-b_2),d) \leq d  \sum_{h}(h,d) \sum_{b_1,b_2} (b_1-b_2,d) \\
& \pprec d HB(d+B) \pprec HB^2 d + HBd^2.
\end{align*}

The sum  over $\ell \neq 0, (\ell,d/e)=1$ has $(f,d/e)=1$ if we denote $f:=(\ell + h(b_1-b_2)d/e,e)$. Hence, the contribution from this part bounded by
\begin{align*}
\sum_{e|d} e \sum_{\substack{0<|\ell| \ll L \\ (\ell,d/e)=1}} \sum_h \sum_{b_1,b_2} &(\ell + h(b_1-b_2)d/e,e) \leq \sum_{\substack{f|e|d \\ (f,d/e)=1}} ef \sum_{\substack{0<|\ell| \ll L \\ (\ell,d/e)=1}}  \sum_{\substack{b_1,b_2}} \sum_h 1_{\ell + h(b_1-b_2)d/e \equiv 0 \,\,(f)} \\
&\leq \sum_{\substack{f|e|d \\ (f,d/e)=1}} ef \sum_{\substack{0<|\ell| \ll L \\ (\ell,d/e)=1}} (\ell,f) \sum_{\substack{b_1,b_2}} (H/f+1) \\
&\pprec B^2\sum_{\substack{f|e|d }}  ef L (H/f+1)
\pprec   H^2B^2 + HB^2 d.
\end{align*}

Combining the contributions from $\ell=0$ and $\ell \neq 0$ we get
\begin{align*}
\sum_{h_1,h_2} \sum_{b_1,b_2} |S(h_1,h_2,b_1,b_2)| \pprec H^2B^2 + HB^2 d + HBd^2,
\end{align*}
so that we have
\begin{align*}
\hat{\Sigma}(M) &\pprec  M^{1/2} \bigg(\bigg(1+ \frac{M}{d^2} \bigg)\frac{1}{H^2} \sum_{h_1,h_2} \sum_{b_1,b_2} |S(h_1,h_2,b_1,b_2)| \bigg)^{1/2} \\
& \pprec M^{1/2} \bigg(\bigg(1+ \frac{M}{d^2} \bigg)\frac{1}{H^2} (H^2B^2 + HB^2 d + HBd^2 )\bigg)^{1/2} \\
&\pprec M^{1/2} B + \frac{M B}{d} + \frac{BM^{1/2} d^{1/2} }{H^{1/2}} +\frac{BM  }{ d^{1/2} H^{1/2}} + \frac{M^{1/2}B^{1/2} d}{H^{1/2}} + \frac{M B^{1/2}}{H^{1/2}} \\
&\pprec  M^{1/2} B + \frac{M B}{d} + \frac{BX^{1/2}  }{d^{1/2}} + \frac{BM^{1/2} X^{1/2}  }{ d^{3/2} } + X^{1/2}B^{1/2} + \frac{M^{1/2} X^{1/2} B^{1/2}}{d} 
\end{align*}
since $H \gg d^2/N$. The last expression is $\ll X^{3/2-3\theta/2-\eta} \ll BX^{1-\eta}/d^2$ provided that
\begin{align*}
\begin{cases} M \ll X^{2-2\theta-\eta}, \\
M\ll X^{1-\theta/2-\eta}, \\
D \ll X^{2/3-\eta}, \\
M \ll X^{1-\theta/2-\eta}, \\
D \ll X^{3/5-\eta}, \quad \quad  \text{and} \\
M \ll X^{3/2-3\theta/2- \eta}.
\end{cases}
\end{align*} 
Here the last term dominates for $\theta\geq 1/2$ so that we get a sufficient bound if $\theta < 3/5$ and $M < X^{3/2-3/2\theta -\eta}$.
 \end{proof}
 
 \begin{remark} \label{typeifactremark} The last two terms correspond to the diagonal contribution from $(h_1,b_1)=(h_2,b_2)$. Hence, we could obtain a longer range for $M$ if we were to use a suitable factorization $d=rq$ and keep a sum over $q$ inside in the Cauchy-Schwarz step to mollify the diagonal contribution (cf. proof of \cite[Proposition 5]{m}, for example). We could also improve on the requirement $\theta < 3/5$ by not completing the sum over $m$ at the Cauchy-Schwarz step but instead using Poisson summation to complete the smoothed sum over $m \sim M$ (note that for $\theta<3/5$ we have $\theta<3/2-3\theta/2$) but then we would have to rely on the Weil bound for Kloosterman sums. Another possible improvement in the arithmetic information would be to show a mixed Type I/II estimate in some range (cf. \cite[Proposition 3]{m}, for example).
 \end{remark}
\section{Type II sums}  \label{typeiisection}
For the sieve we will also need  Type II information, which is given by the following. In contrast to the previous section, here we will need the average over the moduli as well as the assumption that the moduli have a factorisation of a suitably form $d^2=r^2 q^2$.
\begin{prop} \label{prelitypeiiprop} Let $\theta \geq 1/2$, $R,Q\geq 1,$ and $RQ \asymp D=X^\theta$. Suppose that $B \in [X^{1/2-\eta}D^{-1/2},4 X^{1/2}D^{-1/2}],$ and suppose that $R^{1/2} < B X^{-\eta}$. Let $MN=X$ satisfy $X^{\theta/2} \ll N \ll X^{1-\eta} R^{-1}Q^{-2}.$ Then for any divisor bounded functions $\alpha, \beta$ with $\beta$ satisfying the Siegel-Walfisz condition we have
\begin{align*}
\sum_{\substack{r^2 \sim R \\ q^2 \sim Q \\ (r,q)=1 \\ q \in \PP}} \bigg| \sum_{b \sim B} \bigg(\sum_{\substack{m \sim M \\ n \sim N \\ mn \equiv bd+1 \, (r^2q^2) }} \alpha(m) \beta(n) - \frac{1}{\varphi(r^2q^2)}\sum_{\substack{m \sim M \\ n \sim N \\ (mn,r^2q^2)=1  }} \alpha(m) \beta(n) \bigg)\bigg| \ll_C \frac{BX}{\sqrt{D} \log^C X}.
\end{align*}
\end{prop}

\begin{proof}
We use Linnik's dispersion method, that is, first we use Cauchy-Schwarz to smoothen the coefficient $\alpha(m)$. After this we evaluate the smooth sum over $m$ directly and use the large sieve to bound the sum over $n$. The average over $b$ will be key, both to control the diagonal contribution when we Cauchy-Schwarz and so that we can `split' one congruence modulo $r^2$ into two separate congruences modulo $r$ before applying the large sieve (cf. evaluation of $W_0$ below).

The left-hand side is (for some $c_{r,q} \in \{0,1,-1\}$ supported on $(r,q)=1$ and $q \in \PP$)
\begin{align} \nonumber
&\sum_{r^2 \sim R} \sum_m \alpha(m) \bigg( \sum_{q^2 \sim Q} c_{r,q} \sum_{b\sim B} \bigg( \sum_{\substack{n \\ mn \equiv brq+1 \, \, (r^2q^2)}} \beta(n) - \frac{1}{\varphi(r^2q^2)} \sum_{\substack{n \\(mn,r^2q^2)=1}}  \beta(n) \bigg) \bigg) \\ \nonumber
&\ll R^{1/4} M^{1/2} (\log^{O(1)}X) \\ \nonumber &  \hspace{10pt} \cdot \bigg( \sum_{r^2 \sim R} \sum_{m\sim M} \bigg| \sum_{q^2 \sim Q} c_{r,q} \sum_{b\sim B} \bigg( \sum_{\substack{n \\ mn \equiv brq+1 \, \, (r^2q^2)}} \beta(n) - \frac{1}{\varphi(r^2q^2)} \sum_{\substack{n \\(mn,r^2q^2)=1}}  \beta(n) \bigg) \bigg|^2 \bigg)^{1/2} \\ \label{dispersionbound}
& \hspace{20pt}=: R^{1/4} M^{1/2} (\log^{O(1)}X) (W-2V + U)^{1/2}
\end{align}
by Cauchy-Schwarz and expanding the square. We will show that each of $W,V,U$ is equal to
\begin{align}  \label{target}
X_0 + O_C\bigg( \frac{MN^2B^2}{R^{3/2}Q \log^C X}\bigg)
\end{align}
for a certain quantity $X_0$ (cf. evaluation of $U_0$ below). Plugging this into (\ref{dispersionbound}) the main terms cancel and we get a bound
\begin{align*}
\ll_C R^{1/4} M^{1/2} (\log^{O(1)}X) \bigg( \frac{MN^2B^2}{R^{3/2}Q \log^C X}\bigg)^{1/2} \ll_C \frac{BX}{\sqrt{D} \log^{C/2-O(1)} X}
\end{align*}
which is sufficient once we take $C$ large enough.

\subsection*{Evaluation of $U$.} We have
\begin{align*}
U &= \sum_{r^2 \sim R} \sum_{q_1^2, \,  q_2^2 \sim Q} c_{r,q_1} c_{r,q_2} \bigg( \sum_{b \sim B} 1 \bigg)^2 \bigg( \sum_{(n_1,rq_1)=1} \beta(n_1) \bigg) \bigg( \sum_{(n_2,rq_2)=1} \beta(n_2) \bigg) \\
& \hspace{200pt}\cdot \frac{1}{\varphi(r^2)^2\varphi(q_1^2) \varphi(q_2^2)} \sum_{\substack{m \sim M \\ (m,rq_1q_2)=1}} 1 \\
& =: U_0 +U_1,
\end{align*}
where $U_0$ is the part with $q_1 \neq q_2$ and $U_1$ has $q_1=q_2$ (recall that $q_1,q_2$ are primes). Splitting the sum over $(m,rq_1q_2)=1$ into congruence classes $m \equiv t \, (rq_1q_2)$ for $(t,rq_1q_2)=1$ and summing over $m \sim M$ and $t$ we obtain (using $\varphi(r^2)=r \varphi(r)$)
\begin{align*}
U_0 &=  \sum_{r^2 \sim R} \sum_{\substack{q_1^2,  q_2^2 \sim Q\\ q_1\neq q_2}} c_{r,q_1} c_{r,q_2} \bigg( \sum_{b \sim B} 1 \bigg)^2 \bigg( \sum_{(n_1,rq_1)=1} \beta(n_1) \bigg) \bigg( \sum_{(n_2,rq_2)=1} \beta(n_2) \bigg) \frac{M}{\varphi(r)^2r^2q_1^2q_2^2} + \mathcal{O} (E) \\& =: X_0 + O (E),
\end{align*}
where 
\begin{align*}
E = \sum_{r^2 \sim R} \sum_{\substack{q_1^2,   q_2^2 \sim Q\\ q_1\neq q_2}} \bigg( \sum_{b \sim B} 1 \bigg)^2 \bigg( \sum_{(n_1,rq_1)=1} |\beta(n_1)| \bigg) \bigg( \sum_{(n_2,rq_2)=1} |\beta(n_2)| \bigg) \frac{1}{\varphi(r)^2 r q_1 q_2} 
 \pprec \frac{N^2B^2}{R}
\end{align*}
is sufficiently small in view of (\ref{target}). For $U_1$ we get by a trivial bound
\begin{align*}
|U_1| \, \pprec  MN^2 B^2 \sum_{r^2 \sim R} \sum_{\substack{q^2 \sim Q}}  \frac{1}{\varphi(r)^2r^2\varphi(q)^2 q^2} \pprec \frac{M N^2 B^2}{R^{3/2} Q^{3/2}}
\end{align*}
which is more than sufficient.

\subsection*{Evaluation of $V$.} We have
\begin{align*}
V &= \sum_{r^2 \sim R} \sum_{q_1^2, \,  q_2^2 \sim Q} c_{r,q_1} c_{r,q_2} \sum_{b_1,b_2 \sim B}   \sum_{\substack{n_1,n_2 \\ (n_1,rq_1)=1 \\ (n_2,rq_2)=1 }} \beta(n_1)\beta(n_2) \frac{1}{\varphi(r^2)\varphi(q_2^2)} \hspace{-25pt} \sum_{\substack{m \sim M \\ (m,q_2)=1 \\ m \equiv \overline{n_1}(b_1rq_1+1) \, \,(r^2q_1^2)}} \hspace{-25pt}  1  \\
& = X_0 + O \bigg( \frac{MN^2B^2}{R^{3/2}Q X^{\eta}}\bigg)
\end{align*}
by a similar argument as with $U$ (write $V=V_0+V_1$ with $V_1$ corresponding to the part $q_1=q_2$, in $V_0$ expand $(m,q_2)=1$ into $m \equiv t \, (q_2),$ and sum over $m$ and $t$ using $M>X^\eta r^2 q_1^2 q_2^2$).

\subsection*{Evaluation of $W$.} 
We have
\begin{align*}
W=\sum_{r^2 \sim R} \sum_{q_1^2, \,  q_2^2 \sim Q} c_{r,q_1} c_{r,q_2} \sum_{b_1,b_2 \sim B}   \sum_{\substack{n_1,n_2 \\ (n_1,rq_1)=1 \\ (n_2,rq_2)=1 }} \beta(n_1)\beta(n_2) \hspace{-25pt} \sum_{\substack{m \sim M \\  m n_1 \equiv b_1rq_1+1 \, \,(r^2q_1^2) \\  m n_2 \equiv b_2rq_2+1 \, \,(r^2q_2^2)}} \hspace{-25pt}  1  
\end{align*}
Again, write $W=W_0+ W_1$ where $W_1$ has $q_1=q_2$, so that
\begin{align*}
W_1 = \sum_{r^2 \sim R} \sum_{q^2\sim Q} c_{r,q}^2 \sum_{b_1,b_2 \sim B}   \sum_{\substack{n_1,n_2 \\ (n_1,rq)=1 \\ (n_2,rq)=1 }} \beta(n_1)\beta(n_2) \hspace{-25pt} \sum_{\substack{m \sim M \\  m n_1 \equiv b_1rq+1 \, \,(r^2q^2) \\  m n_2 \equiv b_2rq+1 \, \,(r^2q^2)}} \hspace{-25pt}  1 
\end{align*}
From the congruences we get
\begin{align*}
m(n_2-n_1) \equiv rq(b_2-b_1) \quad (r^2q^2),
\end{align*}
which gives us $n_2 \equiv n_1 \, \, (rq)$ (since $mn_j \equiv b_jrq+1 \, \,(r^2q^2)$ implies $(m,rq)=1)$). If $\ell := (n_2-n_1)/(rq),$ then
\begin{align*}
(b_2-b_1)rq \equiv m(n_2-n_1) \equiv m \ell rq \quad (r^2q^2),
\end{align*}
which implies $b_2-b_1 \equiv m \ell \equiv \overline{n_1} \ell \, \, (rq)$. Hence, using $M> r^2q^2, N> rq, B < rq$ we get
\begin{align*}
W_1 &\pprec \sum_{r^2 \sim R} \sum_{q^2\sim Q}     \sum_{\substack{n_1,n_2 \sim N \\ n_2 \equiv n_1 \, \, (rq) }} \, \sum_{\substack{b_1,b_2 \sim B \\ b_2 \equiv b_1+ \ell \overline{n_1} \,\, (rq)}}   \, \sum_{\substack{m \sim M \\  m n_1 \equiv b_1rq+1 \, \,(r^2q^2)}} \hspace{-25pt}  1 \\
& \ll   \sum_{r^2 \sim R} \sum_{q^2\sim Q}  N\bigg(\frac{N}{rq} +1\bigg) B \bigg( \frac{B}{rq}+1\bigg)\bigg( \frac{M}{r^2q^2}+1\bigg) \\
& \ll MN^2 B \sum_{r^2 \sim R} \sum_{q^2\sim Q} \frac{1}{r^3q^3}  \ll \frac{MN^2 B}{RQ} \ll \frac{MN^2B^2}{R^{3/2} Q X^{\eta}}.
\end{align*}
by using $\sqrt{R} \ll BX^{-\eta}$ to get the last bound.

The main term is
\begin{align*}
W_0 =\sum_{r^2 \sim R} \sum_{\substack{q_1^2, \,  q_2^2 \sim Q \\ q_1 \neq q_2 }} c_{r,q_1} c_{r,q_2} \sum_{b_1,b_2 \sim B}   \sum_{\substack{n_1,n_2 \\ (n_1,rq_1)=1 \\ (n_2,rq_2)=1 }} \beta(n_1)\beta(n_2) \hspace{-25pt} \sum_{\substack{m \sim M \\  m n_1 \equiv b_1rq_1+1 \, \,(r^2q_1^2) \\  m n_2 \equiv b_2rq_2+1 \, \,(r^2q_2^2)}} \hspace{-25pt}  1.
\end{align*}
Similarly as with $W_1$, the congruences imply that $n_2 \equiv n_1 \, (r),$ and if $\ell := (n_2-n_1)/r,$ then
\begin{align*}
b_2q_2 -b_1q_1 \equiv  m \ell \equiv\overline{n_1} \ell \,\, (r).
\end{align*}
In fact, the congruences for $mn_1$ and $mn_2$ in the sum over $m$ are equivalent to
\begin{align*}
\begin{cases} & n_2 \equiv n_1 \quad (r), \\
&  b_2q_2 -b_1q_1 \equiv  \overline{n_1} \ell \quad (r), \\
& m n_1 \equiv b_1rq_1+1 \quad (r^2q_1^2), \quad \text{and} \\ 
& m n_2 \equiv b_2rq_2+1 \quad (q_2^2).
 \end{cases}
\end{align*} 
Hence, 
\begin{align*}
W_0 & = \sum_{r^2 \sim R} \sum_{\substack{q_1^2, \,  q_2^2 \sim Q \\ q_1 \neq q_2 }} c_{r,q_1} c_{r,q_2}    \sum_{\substack{n_1,n_2 \\ n_2\equiv n_1 \, \, (r) \\ (n_1,rq_1)=1 \\ (n_2,rq_2)=1 }}  \beta(n_1)\beta(n_2) \sum_{\substack{b_1,b_2 \sim B \\  b_2q_2 -b_1q_1 \equiv  \overline{n_1} \ell \, \, (r) }} \, \sum_{\substack{m \sim M \\  m n_1 \equiv b_1rq_1+1 \, \,(r^2q_1^2) \\  m n_2 \equiv b_2rq_2+1 \, \,(q_2^2)}} \hspace{-25pt}  1 \\
& = \sum_{r^2 \sim R} \sum_{\substack{q_1^2, \,  q_2^2 \sim Q \\ q_1 \neq q_2 }} c_{r,q_1} c_{r,q_2}    \sum_{\substack{n_1,n_2 \\ n_2\equiv n_1 \, \, (r) \\ (n_1,rq_1)=1 \\ (n_2,rq_2)=1 }}  \beta(n_1)\beta(n_2) \sum_{\substack{b_1,b_2 \sim B \\  b_2q_2 -b_1q_1 \equiv  \overline{n_1} \ell \, \, (r) }} \bigg( \frac{M}{r^2q_1^2 q_2^2} + O(1) \bigg).
\end{align*}
Here the error term from the $O(1)$ is negligible, since $M > X^{\eta} r^2q_1^2q_2^2$ and $N,B > r.$ For the main term, since $(r,q_j)=1$ and $B \gg r X^{\eta},$ we get
\begin{align*}
W_0 = M \sum_{r^2 \sim R} \sum_{\substack{q_1^2, \,  q_2^2 \sim Q \\ q_1 \neq q_2 }} c_{r,q_1} c_{r,q_2}     \frac{1}{r^3 q_1^2 q_2^2} \bigg( \sum_{b \sim B} 1 \bigg)^2 \sum_{\substack{n_1,n_2 \\ n_2\equiv n_1 \, \, (r) \\ (n_1,rq_1)=1 \\ (n_2,rq_2)=1 }}  \beta(n_1)\beta(n_2)+ O\bigg( \frac{MN^2B^2}{R^{3/2}Q X^{\eta}}\bigg).
\end{align*}
We expand the congruence $n_2 \equiv n_1 \, \, (r)$ using Dirichlet characters modulo $r$; the principal character gives us exactly $X_0,$ so that we need to bound the error term
\begin{align*}
S =   \frac{M B^2}{R^{3/2}Q^2}\sum_{\substack{q_1, \,  q_2 \sim \sqrt{Q} \\ q_1 \neq q_2 }}      \sum_{r \sim \sqrt{R}} \frac{1}{\varphi(r)} \sum_{\substack{\chi \, \, (r)\\ \chi \neq \chi_0}} \bigg | \sum_{\substack{n_1 \\  (n_1,rq_1)=1 }} \beta(n_1) \chi(n_1)  \bigg |  \bigg | \sum_{\substack{n_2 \\  (n_2,rq_2)=1 }} \beta(n_2) \overline{\chi(n_2)}  \bigg |.
\end{align*}
Since $\beta$ satisfies the Siegel-Walfisz condition, we obtain by Cauchy-Schwarz and Lemma \ref{large}
\begin{align*}
  \sum_{r \sim \sqrt{R}} \frac{1}{\varphi(r)} \sum_{\substack{\chi \, \, (r)\\ \chi \neq \chi_0}} &\bigg | \sum_{\substack{n_1 \\  (n_1,rq_1)=1 }} \beta(n_1) \chi(n_1)  \bigg | \bigg | \sum_{\substack{n_2 \\  (n_2,rq_2)=1 }} \beta(n_2) \overline{\chi(n_2)}  \bigg | 
 \\
 &\ll_C   \bigg( \sqrt{R} + \frac{N}{\log^C X } \bigg) N \log^{O(1)} X \, \ll_C \frac{N^2}{\log^{C-O(1)} X }
\end{align*}
since $N \gg \sqrt{D} = \sqrt{RQ} \gg X^\eta \sqrt{R}$. Thus,
\begin{align*}
S \ll_C \frac{MN^2B^2}{R^{3/2}Q \log^C X}.
\end{align*} 
\end{proof}

\section{The sieve argument for the asymptotic result} \label{sieveasympsection}
\subsection{Reduction to factorable moduli}
To get a handle on Theorem \ref{asymptotictheorem}, we require that $d$ factorizes suitably. More precisely, for any small $\delta >0$ we define $\DD_X(p;\theta,\delta) := \DD_X(p,\theta) \cap \SS_X(\delta),$ where
\begin{align*}
\SS_X(\delta):=\{ r^2q^2: \,\, q \in [X^{\delta^2},X^{\delta})\cap \PP, \,\, (r,q)=1 \}.
\end{align*}
We can relate $|\DD_X(p;\theta)|$ to $|\DD_X(p;\theta, \delta)|$ as follows. By Lemma \ref{kowalskilemma} and by the Brun-Titchmarsh inequality (Lemma \ref{btlemma}) we have
\begin{align*}
\sum_{X < p \leq 2X} |\DD_X(p;\theta)| &= \sum_{X < p \leq 2X} \sum_{|a| < 2 \sqrt{p}} \sum_{\substack{X^\theta < d^2 \leq 2 X^\theta \\ d | a-2 \\ d^2 | p+1-a}} 1 = \sum_{X^\theta < d^2 \leq 2 X^\theta} \sum_{\substack{|a| < 2 \sqrt{2X} \\ a \equiv 2 \,\, (d)}} \sum_{\substack{X < p \leq 2 X \\ p \equiv a-1 \, (d^2) \\ p > a^2/4}} 1 \\
& = \sum_{X < p \leq 2X} |\DD_X(p;\theta,\delta)| + \sum_{ \substack{X^\theta < d^2 \leq 2 X^\theta \\ d^2 \notin \SS_X(\delta) }} \sum_{\substack{|a| < 2 \sqrt{2X} \\ a \equiv 2 \,\, (d)}} \sum_{\substack{X < p \leq 2 X \\ p \equiv a-1 \, (d^2) \\ p > a^2/4}} 1 \\
& =\sum_{X < p \leq 2X} |\DD_X(p;\theta,\delta)| + O\bigg(\sum_{ \substack{X^\theta < d^2 \leq 2 X^\theta \\ (d,P(X^{\delta^2},X^{\delta}))=1 }} \frac{X^{3/2}}{ d \varphi(d^2) \log X}  \bigg),
\end{align*}
where $P(y,z) := \prod_{y \leq p < z} p.$ The proportion of integers which have no prime factors in $[y,z]$ is by a standard sieve bound $\ll \log z / \log y.$ Hence, using $\varphi(d^2)=d \varphi(d)$ and $\sum_{r \sim R} 1/\varphi(r) \ll 1$ we get
\begin{align*}
\sum_{X < p \leq 2X} | \DD_X(p;\theta)|  = \sum_{X < p \leq 2X} |\DD_X(p;\theta,\delta)| + O\bigg( \frac{ \delta X^{3/2-\theta}}{\log X} \bigg),
\end{align*}
so that the error term is smaller than the expected main term by a factor of $\delta.$ Hence, by choosing $\delta^2 = 2 \epsilon = 2(\theta-1/2),$ Theorem \ref{asymptotictheorem} follows from the following result which we will prove below.
\begin{theorem} \label{deltatheorem} For $\theta=1/2+\epsilon$ and $\delta^2 = 2 \epsilon$ we have
\begin{align*}
\sum_{X < p \leq 2X} |\DD_X(p;\theta,\delta)| = \sum_{X^\theta < d^2 \leq 2 X^\theta} \sum_{\substack{|a| < 2 \sqrt{2X} \\ a \equiv 2 \,\, (d)}} \frac{2X- a^2/4}{\varphi(d^2) \log X} + O \bigg( \frac{ \sqrt{\epsilon} X^{3/2-\theta}}{\log X}  \bigg),
\end{align*}
where the implied constant is absolute.
\end{theorem} 

\begin{remark} Let $D:= X^\theta$.  The length of the sum over $|a| < 2 \sqrt{2 X}, \, a \equiv 2 \, (d)$ is of order $\sqrt{X/D}.$ For $d=rq$ with $q \in [X^{\delta^2},X^\delta],$ it is crucial for us that
\begin{align*}
r =d/q \ll \sqrt{D} X^{-2 \epsilon} = X^{-\epsilon} \sqrt{X/D}
\end{align*}
is slightly shorter than the length of the sum over $a$ (cf. proof of Proposition \ref{typeiilowerprop}).
\end{remark}
\subsection{The arithmetic information}
\label{aisection}
In the next section we give the proof of Theorem \ref{deltatheorem}.  In this section we state the required arithmetic information.

We extend the function $p \mapsto |\DD_X(p;\theta,\delta)|$ to all integers by defining
\begin{align*}
|\DD_X(n; \theta, \delta)| :=  \sum_{|a| < 2 \sqrt{n}} \sum_{\substack{X^\theta < d^2 \leq 2 X^\theta \\ d | a-2 \\ d^2 | n+1-a}} 1_{d^2 \in \SS_X(\delta)}.
\end{align*} 
We denote by $M_X(n; \theta, \delta)$ the expected average value of this function, that is, 
\begin{align*}
M_X(n; \theta, \delta) = \sum_{\substack{X^\theta < d^2 \leq 2 X^\theta\\ d^2 \in \SS_X(\delta) \\ (n,d^2)=1}} \frac{1}{\varphi(d^2)} \sum_{\substack{|a| < 2 \sqrt{n} \\ a \equiv 2 \, (d) }} 1 .
\end{align*} 
 
For the sieve we need arithmetical information, given by the following two propositions (to simplify the proof of the first we assume that $\theta <3/5$).
\begin{prop} \label{typeiprop}\emph{(Type I estimate).} Let $\theta \in [1/2, 3/5)$. Let $MN=X$ satisfy $M \ll X^{3/2-3\theta/2-\eta}$. Then for any divisor bounded function $\alpha$ we have
\begin{align*}
\sum_{\substack{m \sim M \\ n \sim X/m}} \alpha(m) |\DD_X(mn; \theta, \delta)| = \sum_{\substack{m \sim M \\ n \sim X/m}} \alpha(m)  M_X(mn; \theta, \delta) + O \bigg( X^{3/2-\theta-\eta} \bigg)
\end{align*}
\end{prop}
\begin{proof}
We reduce the statement to Proposition \ref{prelitypeiprop}. By the definitions of $|\DD_X(mn; \theta, \delta)|$ and $M_X(mn; \theta, \delta)$ the claim is equivalent to
\begin{align*}
 \sum_{\substack{d^2 \sim X^\theta \\d^2 \in \SS_X(\delta)}} \sum_{\substack{|a| < 4 \sqrt{X} \\ a \equiv 2 \,\,(d)}}\sum_{\substack{m \sim M \\ n \sim X/m \\ mn \equiv a-1 \, (d^2) \\ |a| < 2 \sqrt{mn}}} \alpha(m)  = \sum_{\substack{d^2 \sim X^\theta \\d^2 \in \SS_X(\delta)}} \sum_{\substack{|a| < 4 \sqrt{X} \\ a \equiv 2 \,\,(d)}} \frac{1}{\varphi(d^2)}\sum_{\substack{m \sim M \\ n \sim X/m \\ (mn,d^2)=1  \\ |a| < 2 \sqrt{mn}}} \alpha(m) +O( X^{3/2-\theta-\eta}).
\end{align*}
Here we do not need the  assumption that $d^2 \in S_X(\delta)$ and actually are able to show the asymptotic point-wise for each $d$. The range could be slightly improved for large $\theta$ by making use of the well-factorability (cf. Remark \ref{typeifactremark}). The cross-condition $|a| < 2 \sqrt{mn}$ may be removed by using a finer-than-dyadic decomposition to the variables $a$ and $k=mn$  (cf. Section \ref{ftdsection}). Writing $a=bd+2$, the claim is reduced to showing that
\begin{align*}
\sum_{b \sim B} \sum_{\substack{m \sim M \\ n \sim X/m \\ mn \equiv bd+1 \, (d^2) }} \alpha(m)  =  \frac{1}{\varphi(d^2)} \sum_{b \sim B} \sum_{\substack{m \sim M \\ n \sim X/m \\ (mn,d^2)=1 }} \alpha(m) +O\bigg( \frac{B X^{1-\eta}}{d^2} \bigg)
\end{align*}
for all $d^2 \sim D:=X^\theta$ and $B \in [X^{1/2-\eta}D^{-1/2},4 X^{1/2}D^{-1/2}]$ (the contribution from $b$ smaller than $X^{1/2-\eta}D^{-1/2}$ is bounded trivially). 

We still wish to replace the condition $n \sim X/m$ by a smoothed version which is independent of $m$. To do this we first use the finer-than-dyadic decomposition to the variables $m$ and $n$ to replace $n \sim X/m$ by $n \sim N$ for some $N \asymp X/M$ (cf. Section \ref{ftdsection} with $\delta=X^{-\eta}$).  By using a smoothed finer-than-dyadic decomposition (similar to Section \ref{ftdsection}), the condition $n \sim N$ may be replaced by some $C^\infty$-smooth function $\psi_{N^{1-\nu}}(n)$ which is supported in a window of length $N^{1-\nu}$ around $N$ for some sufficiently small $\nu >0$. That is, we have
\begin{align*}
\psi_{N^{1-\nu}}(n)=\psi((n-x)/N^{1-\nu})
\end{align*}
for some $x \asymp N$, where $\psi$ is some compactly supported bounded $C^\infty$-smooth function. The contribution from the edges $N$ and $2N$ is bounded trivially. Since the ranges of $M$ and $N$ are defined up to a factor $X^{\eta}$ for some small $\eta>0$, we may replace $\psi_{N^{1-\nu}}(n)$ by $\psi_N(n)=\psi((n-x)/N)$ for some fixed $\psi$ and $x \asymp N^{1/(1-\nu)}$  if $\nu>0$ is sufficiently small. Thus, Proposition \ref{typeiprop} follows from Proposition \ref{prelitypeiprop}.
\end{proof}
\begin{prop} \label{typeiiprop} \emph{(Type II estimate).} Let $1/2 \leq \theta \leq 1/2+\delta^2-\eta$. Let $MN=X$ satisfy $X^{\theta/2} \ll N \ll X^{1-\theta-3\delta}.$ Then for any divisor bounded functions $\alpha, \beta$ with $\beta$ satisfying the Siegel-Walfisz condition (Definition \ref{swdefinition}) we have
\begin{align*}
\sum_{\substack{m \sim M \\ n \sim N}} \alpha(m) \beta(n)|\DD_X(mn; \theta, \delta)| = \sum_{\substack{m \sim M \\ n \sim N}} \alpha(m) \beta(n) M_X(mn; \theta, \delta) + O_C \bigg( \frac{X^{3/2-\theta}}{\log^C X} \bigg)
\end{align*}
\end{prop}
\begin{remark} Note that for $\theta$ slightly above 1/2 we essentially have Type II information for $M$ and $N$ in the range $[X^{1/4},X^{3/4}]$ except for a small gap in the middle, that is, when $M$ and $N$ are both close to $X^{1/2}$. Without this gap we could apply Vaughan's identity to obtain the desired asymptotic formula for $\sum_{p\sim X} |\DD_X(p; \theta, \delta)| $. For $\theta$ slightly above  1/2 we will apply Harman's sieve to show that the error term coming from this gap is essentially proportional to the width of this gap ($\ll \sqrt{\epsilon}$), which gives the error term in Theorem \ref{deltatheorem}.
 \end{remark}
 \begin{proof}
Similarly to the previous, we reduce the statement to Proposition \ref{prelitypeiiprop}. Under the assumptions of the proposition, we need to show
\begin{align*}
\sum_{\substack{m \sim M \\ n \sim N}} \alpha(m) \beta(n)(|\DD_X(mn; \theta, \delta)| -M_X(mn; \theta, \delta) )\ll_C  \frac{X^{3/2-\theta}}{\log^C X}.
\end{align*}
By the definitions of $|\DD_X(mn; \theta, \delta)|$ and $M_X(mn; \theta, \delta)$ this is equivalent to
\begin{align*}
 \sum_{\substack{d^2 \sim X^\theta \\d^2 \in \SS_X(\delta)}} \sum_{\substack{|a| < 4 \sqrt{X} \\ a \equiv 2 \,\,(d)}} \bigg(\sum_{\substack{m \sim M \\ n \sim N \\ mn \equiv a-1 \, (d^2) \\ |a| < 2 \sqrt{mn}}} \alpha(m) \beta(n) - \frac{1}{\varphi(d^2)}\sum_{\substack{m \sim M \\ n \sim N \\ (mn,d^2)=1  \\ |a| < 2 \sqrt{mn}}} \alpha(m) \beta(n)\bigg) \ll_C  \frac{X^{3/2-\theta}}{\log^C X}.
\end{align*}
The cross-condition $|a| < 2 \sqrt{mn}$ is again easily removed by using a finer-than-dyadic decomposition (cf. Section \ref{ftdsection}). By writing $a=bd+2$, the claim follows once we show that
\begin{align*}
 \sum_{\substack{d^2 \sim X^\theta \\d^2 \in \SS_X(\delta)}} \sum_{\substack{b \sim B}} \bigg(\sum_{\substack{m \sim M \\ n \sim N \\ mn \equiv bd+1 \, (d^2) }} \alpha(m) \beta(n) - \frac{1}{\varphi(d^2)}\sum_{\substack{m \sim M \\ n \sim N \\ (mn,d^2)=1  }} \alpha(m) \beta(n)\bigg) \ll_C  \frac{X^{3/2-\theta}}{\log^C X}
\end{align*}
holds for all $B \in [X^{1/2-\eta}D^{-1/2},4 X^{1/2}D^{-1/2}]$ (the contribution from $b$ smaller than $X^{1/2-\eta}D^{-1/2}$ is bounded trivially). Writing $d^2=r^2q^2$ and splitting the sums over $r$ and $q$ dyadically, we see that Proposition \ref{typeiprop} follows from Proposition \ref{prelitypeiprop}, where $\eta>0$ is chosen to be sufficiently small compared to $\delta>0$.
 \end{proof}
\subsection{Proof of Theorem \ref{deltatheorem}}
Define
\begin{align*}
S(\A_q, z) := \sum_{\substack{X/q < n \leq 2 X/q \\ (n,P(z))=1}} |\DD_X(nq; \theta, \delta)|  \quad \quad \text{and} \quad \quad S(\B_q, z)  :=  \sum_{\substack{X/q < n \leq 2 X/q \\ (n,P(z))=1}} M_X(nq; \theta, \delta),
\end{align*} 
so that (denoting $\A=\A_1$ and $\B=\B_1$) we have 
\begin{align*}
S(\A, 2 \sqrt{X}) =  \sum_{X < p \leq 2X} |\DD_X(p;\theta,\delta)|,
\end{align*}
and $S(\B, 2 \sqrt{X})$ is the expected main term of this sum. Hence, to prove Theorem \ref{deltatheorem} we need to show
\begin{align*}
 S(\A, 2 \sqrt{X}) &\geq  S(\B, 2 \sqrt{X}) - E_1 \quad \quad \text{and} \\
 S(\A, 2 \sqrt{X}) &\leq  S(\B, 2 \sqrt{X}) + E_2
\end{align*}
 for some error terms $E_1,E_2 \ll \sqrt{\epsilon} X^{3/2-\theta}/\log X$.

We first combine the Propositions \ref{typeiprop} and \ref{typeiiprop} to produce an asymptotic formula for sums over almost-primes (a variant of \cite[Theorem 3.1]{harman}).
\begin{prop} \label{funasympprop} Let $Z:=X^{1/4-4\delta}$ and $M \ll X^{3/2-3\theta/2-\eta}$ with $\theta =1/2 +\epsilon$ for some small $\epsilon>0$ and $\delta^2=2\epsilon$. Then for any divisor bounded $\alpha(m)$ we have
\begin{align*}
\sum_{m \sim M} \alpha(m) S(\A_m, Z) = \sum_{m \sim M} \alpha(m) S(\B_m, Z) + O_C \bigg( \frac{X^{3/2-\theta}}{\log^C X} \bigg).
\end{align*}
 \end{prop}
 \begin{proof}
 By using the M\"obius function to expand $(n,P(Z))=1$ we get for $\CC \in \{\A,\B\}$
 \begin{align*}
 \sum_{m \sim M} \alpha(m) S(\CC_m, Z) =  \sum_{m \sim M} \alpha(m) \sum_{d|P(Z)} \mu(d) \sum_{dmn \sim X} F_\CC(dmn),
 \end{align*}
 where $F_\A(n)=|\DD_X(n; \theta, \delta)|$ and $F_\B(n)=M_X(n; \theta, \delta).$  We split both sums in two parts depending on the size of $dm$ to get
\begin{align*}
\sum_{m \sim M} \alpha(m) S(\CC_m, Z) = & \sum_{m \sim M} \alpha(m) \sum_{\substack{d|P(Z) \\ dm \leq X^{\theta+3\delta}}} \mu(d) \sum_{dmn \sim X} F_\CC(dmn) \\
&+\sum_{m \sim M} \alpha(m) \sum_{\substack{d|P(Z) \\ dm > X^{\theta+3\delta}}} \mu(d) \sum_{dmn \sim X} F_\CC(dmn) \\
&=: \Sigma_{\text{I}}(\CC) +\Sigma_{\text{II}}(\CC).
\end{align*}   
 
In the part $dm \leq X^{\theta+3\delta} \leq X^{3/2-3\theta/2-\eta}$ we get an asymptotic formula
\begin{align*}
\Sigma_{\text{I}}(\A) =\Sigma_{\text{I}}(\B) +O_C \bigg( \frac{X^{3/2-\theta}}{\log^C X} \bigg)
\end{align*}
 by Proposition \ref{typeiprop}, since the sums are of the form
\begin{align*}
\sum_{m' \leq X^{\theta+3\delta}} \tilde{\alpha}(m')  \sum_{m'n \sim X} F_\CC(m'n)
\end{align*}
with
\begin{align*}
\tilde{\alpha}(m')  := \sum_{\substack{m' = dm \\ d| P(Z) \\ m \sim M}} \alpha(m)\mu(d).
\end{align*}

In the part $dm > X^{\theta+3\delta}$ write $d=p_1\cdots p_k$ for $p_1 < \cdots < p_k < Z$ to get
\begin{align*}
\sum_{k \ll \log X}(-1)^k \sum_{m\sim M}\alpha(m) \sum_{\substack{p_1 < \cdots < p_k < Z \\ p_1\cdots p_k m > X^{\theta+3\delta}}} \sum_{\substack{n \\ p_1\cdots p_k m n \sim X}} F_\CC(p_1\cdots p_k mn).
\end{align*}
In the summations we have $n \ll X^{1-\theta-3\delta}$ and $dn \gg X/m \gg X^{\theta/2}$. Hence, either $X^{\theta/2} \ll  n \ll  X^{1-\theta-3\delta}$ or there exists a unique $1\leq \ell \leq k$ such that 
\begin{align*}
X^{\theta/2} \ll p_1 \cdots p_\ell n \ll Z X^{\theta/2}\ll X^{1-\theta-3\delta} \quad \text{and} \quad p_1 \cdots p_{\ell-1} n \,\ll X^{\theta/2}.
\end{align*}
The cross-conditions $p_{\ell} < p_{\ell+1}$ and $(p_1\cdots p_\ell) \cdot (p_{\ell+1}\cdots p_k m) >  X^{\theta+3\delta}$ can be removed by using a finer-than-dyadic decomposition (cf. Section \ref{ftdsection}). Therefore, denoting $m'=m p_{\ell+1} \cdots p_k$ and $n'=n p_1 \cdots p_\ell$, we get $\ll \log^{O(1)} X$ sums of the form 
\begin{align*}
\sum_{\substack{m'n' \sim X \\ X^{\theta/2} \ll n' \ll X^{1-\theta-3\delta} }} \tilde{\alpha}(m')\beta(n') F_\CC(p_1\cdots p_k mn),
\end{align*}
where
\begin{align*}
\tilde{\alpha}(m') = \sum_{\substack{m'= p_{\ell+1} \cdots p_k m \sim M'\\ m \sim M\\ p_{\ell+1} < \cdots < p_k < Z\\ p_{\ell+1} \sim P_1}} \alpha(m) \quad \text{and} \quad \beta(n') = \sum_{\substack{n'=p_1 \cdots p_\ell n \\p_1 < \cdots < p_\ell < Z \\p_1 \cdots p_{\ell-1} n \,\ll X^{\theta/2}\\p_\ell \sim P_0 \\ p_1 \cdots p_\ell \sim P_2
}}
\end{align*}
for some dyadic ranges $M',P_0,P_1,P_2$, so that we get an asymptotic formula using Proposition \ref{typeiiprop}.
 Note we have $\ell=0$ and $n'=n$ in the case that $n \gg X^\theta/2$. To see that the coefficients $\beta$ satisfy the Siegel-Walfisz condition (Definition \ref{swdefinition}) we note that this is trivial if $n \gg \exp(\log^{1/2} X)$, and this follows from the Siegel-Walfisz Theorem if $p_\ell \gg \exp(\log^{1/2} X)$. The remaining part is supported on $\exp( \log^{1/2} X)$-smooth numbers which gives a negligible contribution by Lemma \ref{smoothlemma}. Thus, we get\begin{align*}
\Sigma_{\text{II}}(\A) =\Sigma_{\text{II}}(\B) +O_C \bigg( \frac{X^{3/2-\theta}}{\log^C X} \bigg).
\end{align*}
 \end{proof}

Let $Z:=X^{1/4-4\delta}.$ We prove Theorem \ref{deltatheorem} by showing separately the corresponding lower and upper bounds.

For the lower bound we apply Buchstab's identity twice to obtain for $\CC \in \{\A,\B\}$
\begin{align*}
S(\CC, 2 \sqrt{X}) = S(\CC,Z) - \sum_{Z \leq p < 2 \sqrt{X}} S(\CC_p,Z) + \sum_{\substack{Z \leq q < p < 2 \sqrt{X}} } S(\CC_{pq},q),
\end{align*}
For the first two terms we have asymptotic formulas by Proposition \ref{funasympprop}, that is,
\begin{align*}
S(\A,Z)& = S(\B,Z) +  O_C \bigg( \frac{X^{3/2-\theta}}{\log^C X} \bigg)\quad \quad \quad \text{and} \\
\sum_{Z \leq p < 2 \sqrt{X}} S(\A_p,Z) &=\sum_{Z \leq p < 2 \sqrt{X}} S(\B_p,Z) +O_C \bigg( \frac{X^{3/2-\theta}}{\log^C X} \bigg).
\end{align*}
 For the last sum we note that if $q \in [X^{1/4+\epsilon/2},X^{1-\theta-3\delta}],$ then we have a Type II sum, so that we get an asymptotic formula by Proposition \ref{typeiiprop} in this range. Hence, by positivity we have
 \begin{align*}
 S(\A, 2 \sqrt{X}) &= S(\A,Z) - \sum_{Z \leq p < 2 \sqrt{X}} S(\A_p,Z) + \sum_{\substack{Z \leq q < p < 2 \sqrt{X}} } S(\A_{pq},q) \\
  & \geq S(\A,Z) - \sum_{Z \leq p < 2 \sqrt{X}} S(\A_p,Z) + \sum_{\substack{Z \leq q < p < 2 \sqrt{X}\\   X^{1/4+\epsilon/2} \leq q \leq X^{1-\theta-3\delta}} } S(\A_{pq},q) \\
  &= S(\B,Z) - \sum_{Z \leq p < 2 \sqrt{X}} S(\B_p,Z) + \sum_{\substack{Z \leq q < p < 2 \sqrt{X}\\   X^{1/4+\epsilon/2} \leq q \leq X^{1-\theta-3\delta}} } S(\B_{pq},q) +  O_C \bigg( \frac{X^{3/2-\theta}}{\log^C X} \bigg) \\
 & = S(\B, 2 \sqrt{X}) + O_C \bigg( \frac{X^{3/2-\theta}}{\log^C X} \bigg)  - \sum_{\substack{Z \leq q < p < 2 \sqrt{X}\\   q < X^{1/4+\epsilon/2} } } S(\B_{pq},q) -\sum_{\substack{Z \leq q < p < 2 \sqrt{X} \\   q > X^{1-\theta-3\delta}}}  S(\B_{pq},q) .
\end{align*}   
By Lemma \ref{bilemma} the error term from the range $q \leq X^{1/4+\epsilon/2}$ is bounded by $\ll \delta + \epsilon$ times the main term, that is,
\begin{align*}
&\sum_{\substack{Z \leq q < p < 2 \sqrt{X}\\ q < X^{1/4+\epsilon/2}} } S(\B_{pq},q) = \sum_{\substack{Z \leq q < p < 2 \sqrt{X}\\ q < X^{1/4+\epsilon/2}} }  \sum_{\substack{X/pq < n \leq 2 X/pq \\ (n,P(q))=1}} M_X(nq; \theta, \delta) \\
&= \sum_{\substack{Z \leq q < p < 2 \sqrt{X}\\ q < X^{1/4+\epsilon/2}} } \sum_{\substack{X/pq < n \leq 2 X/pq \\ (n,P(q))=1}}\sum_{\substack{X^\theta < d^2 \leq 2 X^\theta\\ d^2 \in \SS_X(\delta) \\ (n,d^2)=1}} \frac{1}{\varphi(d^2)} \sum_{\substack{|a| < 2 \sqrt{n} \\ a \equiv 2 \, (d) }} 1 \\
&=\sum_{\substack{X^\theta < d^2 \leq 2 X^\theta\\ d^2 \in \SS_X(\delta) }} \frac{1}{\varphi(d^2)} \sum_{\substack{|a| < 2 \sqrt{n} \\ a \equiv 2 \, (d) }} \sum_{\substack{Z \leq q < p < 2 \sqrt{X}\\ q < X^{1/4+\epsilon/2}} }  \sum_{\substack{X/pq < n \leq 2 X/pq \\ (n,P(q))=1}}1_{(n,d^2)=1}  \\
&= (1+o(1))\sum_{\substack{X^\theta < d^2 \leq 2 X^\theta\\ d^2 \in \SS_X(\delta) }} \frac{1}{\varphi(d^2)} \sum_{\substack{|a| < 2 \sqrt{n} \\ a \equiv 2 \, (d) }} \frac{X}{\log X} \int_{1/4-4\delta}^{1/4+\epsilon/2} \int_{\beta}^{1/2} \omega \bigg(\frac{1-\alpha-\beta}{\beta} \bigg)\frac{d\alpha d\beta}{\alpha\beta^2}
\\
&\ll (\delta+\epsilon)S(\B, 2 \sqrt{X}).
\end{align*}
For the part where $q>X^{1-\theta-3\delta}$ we have $pq^2 > 4X $ (if $\epsilon>0$ is sufficiently small). Hence, the sum $S(\B_{pq},q)$ is empty except for $q> \sqrt{X}/2$ (since $p \leq 2 \sqrt{X}$), which gives a contribution
\begin{align} \label{buchedgecasebound}
\sum_{\substack{\sqrt{X}/2 \leq q < p < 2 \sqrt{X} }}  S(\B_{pq},q) = \sum_{\substack{\sqrt{X}/2 \leq q < p < 2 \sqrt{X} \\ X < pq \leq 2X
}} M_X(pq; \theta,\delta) \ll \frac{S(\B, 2 \sqrt{X})}{\log X}
\end{align}
by Lemma \ref{bilemma}. Hence, using $\delta^2=2\epsilon$ we get \[S(\A, 2 \sqrt{X}) \geq  S(\B, 2 \sqrt{X}) - E_1\] with $E_1 \ll \sqrt{\epsilon} X^{3/2-\theta}/\log X$.

For the upper bound we apply Buchstab's identity once, which yields for $\CC \in \{\A,\B\}$
\begin{align*}
S(\CC, 2 \sqrt{X}) = S(\CC,Z) - \sum_{Z < p \leq 2 \sqrt{X}} S(\CC_p,p).
\end{align*}
For the first sum we have an asymptotic formula by Proposition \ref{funasympprop}. For $X^{1/4+\epsilon/2} < p < X^{1-\theta-3\delta}$ we have a Type II sum and we get an asymptotic formula by Proposition \ref{typeiiprop}. Thus,
\begin{align*}
S(\A, 2 \sqrt{X}) &\leq S(\B, 2 \sqrt{X}) + O_C \bigg( \frac{X^{3/2-\theta}}{\log^C X} \bigg) + \sum_{\substack{Z \leq p \leq 2 \sqrt{X} \\ p \leq X^{1/4+\epsilon/2} \text{ or } p \geq X^{1-\theta-3\delta}}} S(\B_p,p)
\end{align*}
Similarly as with the lower bound, the error term from the remaining part is by Lemma \ref{bilemma} bounded by   $\ll \delta + \epsilon$ times the main term, so that $S(\A, 2 \sqrt{X}) \leq  S(\B, 2 \sqrt{X}) + E_2$ with $E_2 \ll \sqrt{\epsilon} X^{3/2-\theta}/\log X$. \qed

\section{The sieve argument for the lower bound result} \label{lowerboundsection}
In this section we prove Theorem \ref{lowerboundtheorem}. To this end we define (in view of Proposition \ref{prelitypeiiprop} which requires that $R^{1/2} < B X^{-\eta}$)
\begin{align*}
\D(\theta) = \{r^2 q^2 \sim X^\theta: r \in [X^{1/2-\theta/2-2\eta},X^{1/2-\theta/2-\eta,}] \}
\end{align*}
and set $\DD^{-}_X(p;\theta):=\DD_X(p;\theta) \cap \D(\theta)$ so that $|\DD_X(p;\theta)| \geq |\DD^-_X(p;\theta)|$. Similarly as before, we extend $|\DD^{-}_X(p;\theta)|$ to all integers by defining
\begin{align*}
|\DD^{-}_X(n;\theta)|:=\sum_{|a| < 2 \sqrt{n}} \sum_{\substack{X^\theta < d^2 \leq 2 X^\theta \\ d | n-1 \\ d^2 | n+1-a}} 1_{d^2 \in \D(\theta)}.
\end{align*}
We denote by $M^{-}_X(n; \theta)$ the expected average value of this function, that is, 
\begin{align*}
M^{-}_X(n; \theta) = \sum_{\substack{X^\theta < d^2 \leq 2 X^\theta\\ d^2 \in \D(\theta) \\ (n,d^2)=1}} \frac{1}{\varphi(d^2)} \sum_{\substack{|a| < 2 \sqrt{n} \\ a \equiv 2 \, (d) }} 1.
\end{align*}
By similar technique as in the reductions in Section \ref{aisection}, we see that Propositions \ref{prelitypeiprop} and \ref{prelitypeiiprop} provide us with the following arithmetical information.
\begin{prop}\emph{\textbf{(Type I information)}} \label{typeilowerprop} Suppose $\theta \in [1/2,3/5).$ Let $M \ll X^{3/2-3\theta/2-\eta}$. Then for any divisor bounded function $\alpha$ we have
\begin{align*}
\sum_{\substack{m \sim M \\ n \sim X/m}} \alpha(m) |\DD^-_X(mn; \theta)| = \sum_{\substack{m \sim M \\ n \sim X/m}} \alpha(m)  M^{-}_X(mn; \theta) + O \bigg( X^{3/2-\theta-\eta} \bigg).
\end{align*}
\end{prop}

\begin{prop} \emph{\textbf{(Type II information)}} \label{typeiilowerprop} Suppose $\theta \geq 1/2$. Let $MN=X$ satisfy $X^{\theta/2} \ll N \ll X^{2-3\theta-\eta}$. Then for any divisor bounded functions $\alpha, \beta$ with $\beta$ satisfying the Siegel-Walfisz condition we have
\begin{align*}
\sum_{\substack{m \sim M \\ n \sim N}} \alpha(m) \beta(n)|\DD^-_X(mn; \theta)| = \sum_{\substack{m \sim M \\ n \sim N}} \alpha(m) \beta(n) M^-_X(mn; \theta) + O_C \bigg( \frac{X^{3/2-\theta}}{\log^C X} \bigg)
\end{align*}
\end{prop}

A key parameter in Harman's sieve method is the \emph{width} of the Type II information, which is determined by the exponents in the range of $N$. In our case this is
\begin{align*}
\gamma(\theta):= (2-3\theta -\eta)- \theta/2=2-7\theta/2-\eta
\end{align*}
so we set $Z:=X^{\gamma(\theta)}$. Define
\begin{align*}
S(\A^-_q, z) := \sum_{\substack{X/q < n \leq 2 X/q \\ (n,P(z))=1}} |\DD^-_X(nq; \theta, \delta)|  \quad \quad \text{and} \quad \quad S(\B^-_q, z)  :=  \sum_{\substack{X/q < n \leq 2 X/q \\ (n,P(z))=1}} M^-_X(nq; \theta),
\end{align*} 
so that $S(\A^-, 2 \sqrt{X}) =  \sum_{X < p \leq 2X} |\DD^-_X(p;\theta)|,$ and $S(\B^-, 2 \sqrt{X})$ is the expected main term of this sum. Then Theorem \ref{lowerboundtheorem} follows once we show that $S(\A^-, 2 \sqrt{X}) \gg S(\B^-, 2 \sqrt{X}).$

Similarly as in Proposition \ref{funasympprop}, we first prove an asymptotic formula for almost primes. The argument is exactly same as before, the only difference is that we need to assume $\theta < 5/9$ to guarantee that $3/2-3\theta/2- \eta > 3\theta -1+\eta$, so that all ranges are covered by sums of either Type I or Type II.
\begin{prop} \label{funlowerprop} Suppose that $\theta \in [1/2,5/9).$ Let $Z:=X^{\gamma(\theta)}$ and $M \ll X^{3/2-3\theta/2-\eta}$. Then for any divisor bounded $\alpha(m)$ we have
\begin{align*}
\sum_{m \sim M} \alpha(m) S(\A^-_m, Z) = \sum_{m \sim M} \alpha(m) S(\B^-_m, Z) + O_C \bigg( \frac{X^{3/2-\theta}}{\log^C X} \bigg).
\end{align*}
 \end{prop}
\emph{Proof of Theorem \ref{lowerboundtheorem}}. From here on we let $\theta \in [1/2,0.5388]$ and $\gamma=2-7\theta/2 -\eta \geq 0.1142-\eta$. By two applications of Buchstab's identity we have for $\CC \in \{\A^-,\B^-\}$ and $Z=X^\gamma$
\begin{align*}
S(\CC, 2 \sqrt{X})& = S(\CC,Z) - \sum_{Z \leq p < 2 \sqrt{X}} S(\CC_p,Z) + \sum_{\substack{Z \leq q < p < 2 \sqrt{X}} } S(\CC_{pq},q) 
\end{align*}
In the first two sums we have asymptotic formulas by Proposition \ref{funlowerprop}. In the third sum we have asymptotic formulas in the parts where $p,$ $q$, or $pq$ is in the Type II range $[X^{\theta/2},X^{\theta/2} Z] \cup [X^{1-\theta/2}Z^{-1},X^{1-\theta/2} ]$. Let
\begin{align*}
\U:=\{(u_1,u_2): \,Z \leq & u_2 < u_1 < 2 \sqrt{X},\, u_1 u_2^2 < 4X,  \\\
& u_1,u_2,u_1 u_2 \notin [X^{\theta/2},X^{\theta/2} Z] \cup [X^{1-\theta/2}Z^{-1},X^{1-\theta/2}]\}.
\end{align*}
Then
\begin{align*}
S(\A^-, 2 \sqrt{X})  &\geq  S(\A^-,Z) - \sum_{Z \leq p < 2 \sqrt{X}} S(\A^-_p,Z) + \sum_{\substack{Z \leq q < p < 2 \sqrt{X}\\ p q^2 < 4X \\ (p,q) \notin \U} } S(\A^-_{pq},q)  \\
&= S(\B^-,Z) - \sum_{Z \leq p < 2 \sqrt{X}} S(\B^-_p,Z) + \sum_{\substack{Z \leq q < p < 2 \sqrt{X}\\ p q^2 < 4X \\ (p,q) \notin \U} } S(\B^-_{pq},q) +  O_C \bigg( \frac{X^{3/2-\theta}}{\log^C X} \bigg)\\
&= (1+o(1)) S(\B^-, 2 \sqrt{X}) - \sum_{\substack{ (p,q) \in \U} } S(\B^-_{pq},q)-\sum_{\substack{Z \leq q < p < 2 \sqrt{X}\\ p q^2 \geq 4X }} S(\B^-_{pq},q) \\
& \geq (0.02-o(1)) S(\B^-, 2 \sqrt{X}),
\end{align*}
since by Lemma \ref{bilemma}
\begin{align*}
\sum_{\substack{ (p,q) \in \U} } S(\B^-_{pq},q)  \leq (c_0+o(1))S(\B, 2 \sqrt{X})
\end{align*}
for
\begin{align*}
c_0 = \int_{(X^{\alpha},X^\beta) \in \U} \omega \bigg(\frac{1-\alpha-\beta}{\beta} \bigg)\frac{d\alpha d\beta}{\alpha\beta^2} < 0.98.
\end{align*}
The code used to compute this can be found at \url{http://codepad.org/CePKg8aA}.
Note also that by the same argument as in (\ref{buchedgecasebound}) we have
\begin{align*}
\sum_{\substack{Z \leq q < p < 2 \sqrt{X}\\ p q^2 \geq 4X }} S(\B^-_{pq},q) \ll \frac{S(\B^-, 2 \sqrt{X})}{\log X},
\end{align*}
which is negligible.
\qed
\begin{remark} Here we could improve the exponent by iterating Buchstab's identity on some parts of the sum over $(p,q)\in \U$ to generate more Type II sums. However, since it is likely that the arithmetic information can be improved with available methods it is not worthwhile to spend a lot of effort optimizing the sieve argument at this stage.
\end{remark}
\section{Almost all moduli} \label{almostallsection}
In this section we explain how to modify the arguments in the previous sections to obtain Theorems \ref{almostallmodulitheorem} and \ref{almostall2theorem}. We simply have to note that the arithmetic information holds either for all moduli $d$ (Proposition \ref{prelitypeiprop}) or for all but a proportion $O(\log^{-C} X)$ of moduli $d$ with a suitable factorization property (Proposition \ref{prelitypeiiprop}, note that the restriction $(r,q)=1$ is easily removed since $q$ is prime and there are very few integers that are divisible by $q^2$). This means that the sieve arguments given in Section \ref{lowerboundsection} hold also individually for all but a proportion $O(\log^{-C} X)$ of the moduli $d \in  \D_\eta(0.2694,X^{0.2694})$, which is sufficient to prove  Theorem \ref{almostall2theorem}. To prove Theorem \ref{almostallmodulitheorem} it suffices to note that all but a proportion $O(\sqrt{\epsilon})$ of integers $d \sim X^{1/4+\epsilon}$ have the required factorization (cf. the argument used to reduce Theorem \ref{asymptotictheorem} to Theorem \ref{deltatheorem}), so that the sieve argument in  Section \ref{sieveasympsection} may be carried out for all but a proportion $O(\sqrt{\epsilon})$ of moduli. 

\bibliography{ellipticbibl}

\begin{thebibliography}{10}

\bibitem{bz}
S.~Baier and L.~Zhao.
\newblock Bombieri-{V}inogradov type theorems for sparse sets of moduli.
\newblock {\em Acta Arith.}, 125(2):187--201, 2006.

\bibitem{bpfs}
W.~D. Banks, F.~Pappalardi, and I.~E. Shparlinski.
\newblock On group structures realized by elliptic curves over arbitrary finite
  fields.
\newblock {\em Exp. Math.}, 21(1):11--25, 2012.

\bibitem{BFI}
E.~Bombieri, J.~B. Friedlander, and H.~Iwaniec.
\newblock Primes in arithmetic progressions to large moduli.
\newblock {\em Acta Math.}, 156(3-4):203--251, 1986.

\bibitem{cdks}
V.~Chandee, C.~David, D.~Koukoulopoulos, and E.~Smith.
\newblock Group structures of elliptic curves over finite fields.
\newblock {\em Int. Math. Res. Not. IMRN}, (19):5230--5248, 2014.

\bibitem{frs}
R.~R. Farashahi and I.~E. Shparlinski.
\newblock On group structures realized by elliptic curves over a fixed finite
  field.
\newblock {\em Exp. Math.}, 21(1):1--10, 2012.

\bibitem{FI}
E.~Fouvry and H.~Iwaniec.
\newblock Primes in arithmetic progressions.
\newblock {\em Acta Arith.}, 42(2):197--218, 1983.

\bibitem{odc}
J.~Friedlander and H.~Iwaniec.
\newblock {\em Opera de cribro}, volume~57 of {\em American Mathematical
  Society Colloquium Publications}.
\newblock American Mathematical Society, Providence, RI, 2010.

\bibitem{harman}
G.~Harman.
\newblock {\em Prime-detecting sieves}, volume~33 of {\em London Mathematical
  Society Monographs Series}.
\newblock Princeton University Press, Princeton, NJ, 2007.

\bibitem{hbjia}
D.~R. Heath-Brown and C.~Jia.
\newblock The largest prime factor of the integers in an interval. {II}.
\newblock {\em J. Reine Angew. Math.}, 498:35--59, 1998.

\bibitem{kowalski}
E.~Kowalski.
\newblock Analytic problems for elliptic curves.
\newblock {\em J. Ramanujan Math. Soc.}, 21(1):19--114, 2006.

\bibitem{matomaki}
K.~Matom\"{a}ki.
\newblock A note on primes of the form {$p=aq^2+1$}.
\newblock {\em Acta Arith.}, 137(2):133--137, 2009.

\bibitem{m}
J.~Merikoski.
\newblock On the greatest square divisor of shifted primes.
\newblock {\em To appear in Acta Arith.}, 2020.

\bibitem{polymath}
D.~H.~J. Polymath.
\newblock New equidistribution estimates of {Z}hang type.
\newblock {\em Algebra Number Theory}, 8(9):2067--2199, 2014.

\bibitem{Ten}
G.~Tenenbaum.
\newblock {\em Introduction to analytic and probabilistic number theory},
  volume 163 of {\em Graduate Studies in Mathematics}.
\newblock American Mathematical Society, Providence, RI, third edition, 2015.
\newblock Translated from the 2008 French edition by Patrick D. F. Ion.

\bibitem{zhang}
Y.~Zhang.
\newblock Bounded gaps between primes.
\newblock {\em Ann. of Math. (2)}, 179(3):1121--1174, 2014.

\end{thebibliography}
\bibliographystyle{abbrv}
\end{document}